\documentclass[12pt]{article}
\oddsidemargin \evensidemargin
\usepackage{standalone}

\usepackage[english]{babel}
\usepackage[latin1]{inputenc}
\usepackage[T1]{fontenc}
\usepackage{amssymb,amsmath,amstext,amsfonts,amsthm,amscd,latexsym, mathrsfs,amsbsy}
\usepackage{mathtools}
\usepackage{verbatim}
\usepackage{fancyvrb}
\usepackage{relsize}
\usepackage[colorlinks=true, urlcolor=blue, linkcolor=blue, citecolor=blue]{hyperref}
\usepackage{tikz-cd}
\usepackage[normalem]{ulem} 
\usepackage{caption}
\usepackage{graphicx}
\usepackage{array}
\usepackage{arydshln}
\usepackage{xcolor}
\usepackage[matrix,arrow]{xy}
\usepackage{booktabs}
\usepackage{cleveref}
\usepackage{hhline}
\usepackage{cprotect} 
\usepackage{enumitem,kantlipsum}
\usepackage{tcolorbox}

\newtheorem{theorem}{Theorem}[section]

\newtheorem{proposition}[theorem]{Proposition}

\theoremstyle{definition}
\newtheorem{definition}[theorem]{Definition}

\newtheorem{remark}[theorem]{Remark}

\newtheorem{cor}[theorem]{Corollary}
\newtheorem{examplex}[theorem]{Example}
\newenvironment{example}
{\pushQED{\qed}\examplex}
{\popQED\endexamplex}
\newtheorem{exampletwo}[theorem]{Example}

\newcommand*{\mge}[1][]{\mathsmaller{\ge #1}}

\newcommand{\C}{{\mathbb C}}
\newcommand{\G}{{\mathbb G}}

\newcommand{\R}{{\mathbb R}}
\newcommand{\Z}{{\mathbb Z}}

\newcommand{\PP}{{\mathbb P}}

\def\bj{{\boldsymbol{j}}}

\def\bd{{\boldsymbol{d}}}

\def\bpi{{\boldsymbol{\pi}}}
\def\bone{{\boldsymbol{1}}}

\def\codim{\operatorname{codim}}
\def\rank{\operatorname{rank}}
\def\im{\operatorname{im}}

\newcommand{\mT}{\mathsmaller{\mathsf{T}}}

\newcommand{\kn}{{\mathcal N}}
\newcommand{\ko}{{\mathcal O}}

\newcommand{\kz}{{\mathcal Z}}

\makeatletter
\renewcommand*\env@matrix[1][*\c@MaxMatrixCols c]{%
    \hskip -\arraycolsep
    \let\@ifnextchar\new@ifnextchar
    \array{#1}}
\makeatother

\newcommand\blfootnote[1]{%
  \begingroup
  \renewcommand\thefootnote{}\footnote{#1}%
  \addtocounter{footnote}{-1}%
  \endgroup
}

\title{\bf Nash Loci}

\author{Luca Sodomaco and Julian Weigert}

\date{}

\begin{document}

\maketitle

\begin{abstract}
\noindent In the study of Nash equilibria of finite-player games, one often seeks equilibria that are compatible with predetermined constraints, either determined by the players or by an external agent.
We discuss the algebraic loci, called {\em Nash loci}, of games whose {\em Nash equilibrium scheme} intersects a fixed algebraic variety in a product of projective spaces. We determine their dimensions and multidegrees in multiprojective space, and their equations for two-player games and for small multiple-player games. The multilinear equations defining the Nash equilibrium scheme allow us to describe Nash loci in the language of Grassmannians and Pl\"ucker coordinates. Motivated by this fact, we relate Nash loci to {\em multigraded associated varieties}, which are subvarieties in products of Grassmannians that generalize the multigraded Cayley-Chow hypersurfaces of Osserman and Trager.
\end{abstract}

\blfootnote{{\footnotesize 2020 Mathematics Subject Classification: 14A10, 14C17, 14M15, 91A05, 91A06, 91A12, 91A80.\newline Keywords and phrases: Associated variety, game, Grassmannian, scheme, tensor, totally mixed Nash equilibrium, vector bundle.}}

\section{Introduction}\label{sec: intro}

Consider the classical ``Rock, Paper and Scissors'' game, a two-player game encoded by the following payoff table:
\begin{table}[ht]
    \centering
    \begin{tabular}{c|ccc}
           & R & P & S \\\hline
         R & $(0,0)$ & $(-1,1)$ & $(1,-1)$\\
         P & $(1,-1)$ & $(0,0)$ & $(-1,1)$\\
         S & $(-1,1)$ & $(1,-1)$ & $(0,0)$\\
    \end{tabular}
    \caption{The payoff table of the Rock, Paper and Scissors game.}
    \label{tab: RPS}
\end{table}

This game falls into the family of two-player games with a single mixed Nash equilibrium. Let $\pi^{(1)}=(\pi_1^{(1)},\pi_2^{(1)},\pi_3^{(1)})$ and $\pi^{(2)}=(\pi_1^{(2)},\pi_2^{(2)},\pi_3^{(2)})$ denote the strategies of players 1 and 2, respectively, where $\pi^{(1)}$ and $\pi^{(2)}$ are vectors in the two-dimensional probability simplex $\Delta_2\subset\R^3$. The expected payoffs for the two players are, respectively
\begin{align*}
    \Pi_1(\pi^{(1)},\pi^{(2)}) &= -\pi_1^{(1)}\pi_2^{(2)}+\pi_1^{(1)}\pi_3^{(2)}+\pi_2^{(1)}\pi_1^{(2)}-\pi_2^{(1)}\pi_3^{(2)}-\pi_3^{(1)}\pi_1^{(2)}+\pi_3^{(1)}\pi_2^{(2)}\\
    \Pi_2(\pi^{(1)},\pi^{(2)}) &= \pi_1^{(1)}\pi_2^{(2)}-\pi_1^{(1)}\pi_3^{(2)}-\pi_2^{(1)}\pi_1^{(2)}+\pi_2^{(1)}\pi_3^{(2)}+\pi_3^{(1)}\pi_1^{(2)}-\pi_3^{(1)}\pi_2^{(2)}\,,
\end{align*}
in particular $\Pi_1+\Pi_2=0$ due to the shape of the payoff table in Table~\ref{tab: RPS}. The Nash equilibria of the Rock, Paper, and Scissors game are the solutions $(\bar{\pi}^{(1)},\bar{\pi}^{(2)})\in\Delta_2\times\Delta_2$ of the simultaneous optimization problem
\begin{align*}
    \Pi_1\left(\bar{\pi}^{(1)},\bar{\pi}^{(2)}\right) \ =\  \max_{\pi^{(1)}\in\Delta_2}\ \Pi_1\left(\pi^{(1)},\bar{\pi}^{(2)}\right)\,,\quad
    \Pi_2\left(\bar{\pi}^{(1)},\bar{\pi}^{(2)}\right) \ =\  \max_{\pi^{(2)}\in\Delta_2}\ \Pi_2\left(\bar{\pi}^{(1)},\pi^{(2)}\right)\,.
\end{align*}
Using this setting, one verifies that the unique Nash equilibrium of this game is
\begin{equation}\label{eq: tmNE RPS}
    \bar{\pi}^{(1)} = \bar{\pi}^{(2)} = \left(\frac{1}{3},\frac{1}{3},\frac{1}{3}\right)\,,
\end{equation}
in particular it is {\em totally mixed}, namely all three components of the vectors $\bar{\pi}^{(1)}$ and $\bar{\pi}^{(2)}$ have strictly positive values. In other words, the pair $(\bar{\pi}^{(1)},\bar{\pi}^{(2)})$ lies in the product $\Delta_2^\circ\times\Delta_2^\circ$ of the relative interiors of the probability simplices.

\medskip

A feature of the game described above is that the components $\bar{\pi}^{(1)}$ and $\bar{\pi}^{(2)}$ are {\em equal}. This fact motivates a natural question: what are all $3\times 3$ games with at least one Nash equilibrium with equal components?
In this paper, we restrict our attention to {\em totally mixed Nash equilibria}. Phrased in more generality, a possible question could be: given the $n$-player game format $d^{\times n}$, where $d$ is the number of possible choices for each player, what is the locus of all games with that format having at least one totally mixed Nash equilibrium with all equal components?

\medskip

There is nothing special in considering the condition that all components of a totally mixed Nash equilibrium are equal.
In Section~\ref{sec: definition} we introduce the main notations and definitions. We consider {\em multihomogeneous} polynomial constraints in the {\em multiprojective strategy space} $\PP^\bd=\PP^{d_1-1}\times\cdots\times\PP^{d_n-1}$, whose elements are $n$-tuples of vectors encoding mixed strategies of the $n$ players up to a scalar factor. Equivalently, we fix an algebraic subvariety $Y\subseteq\PP^\bd$ and look for all games with a fixed format $\bd=(d_1,\ldots,d_n)$ whose {\em Nash equilibrium scheme} (that is, the locus in $\PP^\bd$ cut out by the equations defining totally mixed Nash equilibria, see Definition~\ref{def: Nash equilibria scheme}) meets $Y$ in at least one point. In Definition~\ref{def: Nash locus} we introduce this locus as the {\em Nash locus} of $Y$.

\medskip

In Section~\ref{sec: codim degree} we compute the codimension and the multidegree of an arbitrary Nash locus using classical intersection theory in multiprojective spaces. In particular, Theorem~\ref{thm: codim degree Nash locus} can be thought of as a ``Nash version'' of \cite[Theorem 1]{shahidi2021degrees} in the context of best-rank-one approximation of higher-order tensors; see also \cite[Section 2.4]{abo2026vector} for a more detailed comparison between totally mixed Nash equilibria and {\em singular vector tuples} of higher-order tensors. In the rest of the section, we give determinantal equations for the Nash loci of two-player games, as well as for some multi-player games.

\medskip

Section~\ref{sec: multigraded associated varieties} is motivated by the following fact. The equations that cut out the Nash equilibrium scheme of an $n$-player game are {\em multilinear} in the multiprojective strategy space $\PP^\bd$. After applying a {\em mixed Segre embedding} described in \eqref{eq: segre}, these equations become {\em linear} in $n$ separate vectors of variables. Equivalently, the Nash equilibrium scheme can be viewed as the intersection between a product of linear spaces and the mixed Segre embedding of the multiprojective strategy space $\PP^\bd$. Similarly, the Nash locus of $Y\subseteq\PP^\bd$ can be interpreted as the locus of products of linear spaces (with the same dimensions as in the unconstrained case) that intersect the mixed Segre embedding of $Y$. Since products of linear spaces are naturally encoded by tuples of points in a product of Grassmannians, the whole language introduced in Section~\ref{sec: codim degree} can now be rewritten in the language of Grassmannians, and the equations of Nash loci can be written as equations in $n$ sets of Pl\"ucker coordinates. Beyond Nash loci and the above observation, in Section~\ref{sec: multigraded associated varieties} we introduce {\em multigraded associated varieties}, which are subvarieties of a product of Grassmannians. Our construction fits naturally in the recent line of research on {\em multigraded Cayley-Chow forms} \cite{osserman2019multigraded} and {\em multigraded Hurwitz forms} \cite{pratt2026multigraded}.

\medskip

We also implemented several \texttt{Macaulay2} \cite{grayson1997macaulay2} functions to compute Nash loci, their degrees, and the degrees of the corresponding multigraded associated varieties. All these functions are gathered in the file {\tt NashLoci.m2}, which is publicly available on {\tt Zenodo}~\cite{sodomaco2026supplementary} together with extensive documentation. Throughout the paper, we provide several computational examples that utilize our software. In each script displayed, we assume the user has already loaded the file {\tt NashLoci.m2} with the command {\tt load "nashLoci.m2"}.

\section{Definition and first examples of Nash Loci}\label{sec: definition}

Our first goal is to introduce Nash loci in Definition~\ref{def: Nash locus} and give some initial computational examples. Before that, we introduce some preliminary notation.

\medskip

Most of the notation used is compatible with \cite{abo2026vector}. Let $n\ge 2$ be the number of players in the game. For any integer $a\ge 1$, we denote by $[a]$ the set $\{1,\ldots,a\}$. For each $i\in[n]$, Player $i$ is allowed to choose among $d_i\ge 2$ pure strategies. When Player $1$ chooses strategy $j_1\in[d_1]$, Player $2$ chooses strategy $j_2\in[d_2]$,\ldots, Player $n$ chooses strategy $j_n\in[d_n]$, then Player $1$ receives a payoff $g_{\bj}^{(1)}=g_{j_1\cdots j_n}^{(1)}\in\R$, Player $2$ receives a payoff $g_{\bj}^{(2)}\in\R$,\ldots, Player $n$ receives a payoff $g_{\bj}^{(n)}\in\R$.
Let $\bd\coloneqq(d_1,\ldots,d_n)\in\Z_{\mge[2]}^n$ be the {\em game format}. We utilize the index set $I\coloneqq\prod_{i=1}^n[d_i]$.
We denote by $\bone$ a tuple that consists of ones and by $\bone_i$ a tuple that has a single zero in the $i$th element and one everywhere else. For simplicity, we use the same notation for the corresponding column vectors.

\medskip

Even though Nash equilibria are defined over $\R$, in what follows we consider complex vector spaces and the solution sets of polynomial systems over $\C$. For every $i\in[n]$, let $V_i\coloneqq\C^{d_i}$ with the standard basis $\{e_1^{(i)},\ldots,e_{d_i}^{(i)}\}$, and consider the tensor product $V\coloneqq\bigotimes_{i=1}^n V_i$. For each $i \in [n]$, let  $(\pi_1^{(i)},\ldots, \pi_{d_i}^{(i)})$ be coordinates on $V_i$, which form the dual basis for the dual space $V_i^*$ of $V_i$ with respect to $\{e_1^{(i)},\ldots,e_{d_i}^{(i)}\}$. We denote by $\pi^{(i)}$ the column vector $(\pi_1^{(i)},\ldots, \pi_{d_i}^{(i)})^\mT$, and let $\bpi\coloneqq(\pi^{(1)},\ldots,\pi^{(n)})\in \prod_{i=1}^n V_i^*$.
We consider the polynomial ring $R_i\coloneqq\C[\pi_1^{(i)},\ldots, \pi_{d_i}^{(i)}]$ for all $i\in[n]$ and $R\coloneqq R_1\otimes\cdots\otimes R_n$. Similarly we consider polynomial rings $S_i\coloneqq\C[\{g_\bj^{(i)} \mid \bj\in I\}]$ for all $i\in[n]$ and $S\coloneqq S_1\otimes\cdots\otimes S_n$.
Finally, for each $i \in [n]$, let $\PP^{d_i-1}\coloneqq \PP(V_i)$ be the projective space of one-dimensional subspaces of $V_i$, and following \cite{abo2026vector} we define $\PP^\bd\coloneqq\prod_{i=1}^n\PP^{d_i-1}$. Here we call it the {\em multiprojective strategy space} associated with $\bd$.

\bigskip

We encode an {\em $n$-player game} $G$ by the tuple of tensors $G=(G^{(1)},\ldots,G^{(n)})$ of format $\bd$ with real entries, namely $G^{(i)}=(g_\bj^{(i)})_{\bj\in I}\in V$ for all $i\in[n]$.
For all $i\in[n]$, let $\Delta_{d_i-1}$ be the $(d_i-1)$-dimensional probability simplex:
\[
    \Delta_i \coloneqq \left\{\pi^{(i)}\in \R^{d_i}\ \left|\ \text{$\pi_1^{(i)}+\cdots+\pi_{d_i}^{(i)}=1$ and $\pi_{j}^{(i)} \ge 0$ for all $j\in[d_i]$}\right.\right\}\,.
\]
A vector $\pi^{(i)}\in \Delta_{d_i-1}$ is called a {\em mixed strategy} of Player $i$ because each component $\pi_{j}^{(i)}$ corresponds to the probability that Player $i$ unilaterally selects the pure strategy $j \in [d_i]$. The $n$ players choose a joint probability distribution $\pi^{(1)}\otimes\cdots\otimes\pi^{(n)}=(\pi_{j_1}^{(1)}\cdots\pi_{j_n}^{(n)})_{\bj\in I}$. The {\em expected payoff} for Player~$i$ is the contraction of the tensors $\pi^{(1)}\otimes\cdots\otimes\pi^{(n)}$ and $G^{(i)}$, namely, 
\begin{equation}\label{eq: expected payoff}
\Pi_i\coloneqq G^{(i)}\cdot\pi^{(1)}\otimes\cdots\otimes\pi^{(n)} = \sum_{\bj\in I}\pi_{j_1}^{(1)}\cdots\pi_{j_n}^{(n)}g_{\bj}^{(i)}\,.
\end{equation}
Note that $\Pi_i=\Pi_i(\pi^{(1)},\ldots,\pi^{(n)})$ is a multilinear function of the $n$ mixed strategies.

\begin{definition}[{\cite{nash1950equilibrium}}]\label{def: Nash equilibrium}
Let $G$ be an $n$-player game in normal form. A {\em Nash equilibrium} of $G$ is a mixed strategy $\bar{\bpi}=(\bar{\pi}^{(1)},\ldots,\bar{\pi}^{(n)})\in\prod_{i=1}^n \Delta_{d_i-1}$ such that, for every $i\in[n]$, Player $i$ cannot increase their expected payoff $\Pi_i$ by changing their mixed strategy $\pi^{(i)}$ while the other players keep their mixed strategies fixed. Equivalently 
\begin{equation}\label{eq: opt NE}
    \Pi_i(\bar{\bpi}) = \max_{\pi^{(i)}\in\Delta_{d_i-1}}\Pi_i(\bar{\pi}^{(1)},\ldots,\bar{\pi}^{(i-1)},\pi^{(i)},\bar{\pi}^{(i+1)},\ldots,\bar{\pi}^{(n)})\quad\forall\,i\in[n]\,.
\end{equation}
A Nash equilibrium $\bar{\bpi}$ is {\em pure} if $\bar{\pi}^{(i)}$ is a vertex of $\Delta_{d_i-1}$ for all $i\in[n]$ and is {\em mixed} otherwise. It is {\em totally mixed} if $\bar{\pi}^{(i)}\in\Delta_{d_i-1}^\circ$ for all $i\in[n]$, where $\Delta_{d_i-1}^\circ$ denotes the relative interior of $\Delta_{d_i-1}$. 
\end{definition}

In this paper, we restrict to totally mixed Nash equilibria. In this scenario, one may first study the complex critical points $\bar{\bpi}=(\bar{\pi}^{(1)},\ldots,\bar{\pi}^{(n)})\in\prod_{i=1}^n V_i$ of the simultaneous optimization problem \eqref{eq: opt NE} with active constraints $\sum_{j=1}^{d_i}\pi_j^{(i)}=1$, and then check which solutions have all components with strictly positive coordinates. Note that $\bar{\bpi}=(\bar{\pi}^{(1)},\ldots,\bar{\pi}^{(n)})\in\prod_{i=1}^n V_i$ is critical for the above optimization problem with constraints $\bar{\pi}^{(i)}\in\Delta_{d_i-1}^\circ$ for all $i\in[n]$ if and only if the gradient vectors of $\Pi_i$ and of $\sum_{j=1}^{d_i}\pi_j^{(i)}-1$ with respect to $\pi^{(i)}$ are proportional. These gradient vectors are respectively
\begin{align*}
    &\nabla_i\Pi_i = \left(\frac{\partial\Pi_i}{\partial \pi_{j_i}^{(i)}}\right)_{j_i=1}^{d_i} = \Bigg(\sum_{\substack{j_k\in[d_k]\\k\neq i}}g_{\bj}^{(i)}\pi_{j_1}^{(1)}\cdots\pi_{j_{i-1}}^{(i-1)}\pi_{j_{i+1}}^{(i+1)}\cdots\pi_{j_n}^{(n)}\Bigg)_{j_i=1}^{d_i} = G^{(i)}\cdot\otimes_{k\neq i}\pi^{(k)}\\
    &\nabla_i\left(\sum_{j_i=1}^{d_i}\pi_{j_i}^{(i)}-1\right) = \bone\,.
\end{align*}
The linear dependence of $\nabla_i\Pi_i$ and $\bone$ is equivalent to imposing that the rank of the $d_i \times 2$ matrix $\begin{pmatrix}\nabla_i\Pi_i &\bone\end{pmatrix}$ is at most one.
For all $n \ge 2$, the conditions $\rank\begin{pmatrix}\nabla_i\Pi_i &\bone\end{pmatrix}\le 1$ yield a system of multihomogeneous polynomials in the $n$ vectors $\pi^{(i)}$. Thus, we are studying the solutions of a polynomial system in the multiprojective strategy space $\PP^\bd$. For all $i\in[n]$ and for all indices $r,s$ in $[d_i]$ with $r<s$, let $f_{r,s}^{(i)}$ be the $2\times 2$ minor of $\begin{pmatrix}\nabla_i\Pi_i &\bone\end{pmatrix}$ obtained by selecting the rows $r$ and $s$. In particular $f_{r,s}^{(i)}=f_{1,s}^{(i)}-f_{1,r}^{(i)}$. For this reason, we have that $\rank\begin{pmatrix}\nabla_i\Pi_i &\bone\end{pmatrix}\le 1$ if and only if $\sigma_k^{(i)}\coloneqq f_{1,k+1}^{(i)}=0$ for all $k\in[d_i-1]$ and $i\in[n]$. This leads to the following definition.

\begin{definition}[{\cite[Definition 2.6]{abo2026vector}}]\label{def: Nash equilibria scheme}
Let $G=(G^{(1)},\ldots,G^{(n)})\in V^{\oplus n}$. For all $i \in [n]$, let $J_i$ be the ideal in $R$ generated by $\sigma_1^{(i)},\ldots,\sigma_{d_i-1}^{(i)}$.
The {\em Nash equilibrium scheme} of $G$ is the subscheme $\kz_G\subseteq\PP^\bd$ defined by the ideal $J\coloneqq \sum_{i=1}^n J_i$.
\end{definition}

It follows immediately that $\bpi=(\pi^{(1)},\ldots,\pi^{(n)}) \in \prod_{i=1}^{d_i} \Delta_{d_i-1}^\circ$ is a totally mixed Nash equilibrium of the game $G=(G^{(1)},\ldots,G^{(n)})$ if the corresponding point~$[\bpi]\coloneqq([\pi^{(1)}],\ldots,[\pi^{(n)}])$ of $\PP^\bd$ lies in the Nash equilibrium scheme $\kz_G$.

\medskip

In this paper, we frequently utilize tools from intersection theory and vector bundles over multiprojective spaces. Again, our notation is consistent with \cite[Section 2.3]{abo2026vector}.
The motivation to use these tools comes from the natural observation that $\kz_G$ is defined by multilinear polynomials in $n$ sets of homogeneous variables. Said formally, $\kz_G$ is the zero scheme of a global section $\sigma=(\sigma^{(1)},\ldots,\sigma^{(n)})\in H^0(\PP^\bd,E)$, where
\begin{equation}\label{eq: def vector bundle E}
    E=E_1\oplus\cdots\oplus E_n\,,\qquad E_i\coloneqq\ko_{\PP^{\bd}}(\bone_i)^{\oplus(d_i-1)}\quad\forall\,i\in[n]\,,
\end{equation}
a vector bundle of rank equal to $\dim\PP^{\bd}=\sum_{i=1}^n(d_i-1)$. When the global section $\sigma$ is generic, then its zero scheme $\kz_G$ is zero-dimensional and reduced, with cardinality given by the degree of the top Chern class of $E$, denoted by $c(\bd)$ in \cite[Theorem 2.8]{abo2026vector}.

The equations $\sigma_k^{(i)}$ used to define $\kz_G$ are multi-homogeneous in the entries of the game $G$: For a second game $G'$ such that there exist nonzero constants $\lambda_1,\ldots,\lambda_n\in \C^*$ with $G'=(\lambda_1 G^{(1)},\ldots,\lambda_nG^{(n)})$ we have $\kz_{G'}=\kz_G$. We therefore write $\kz_{[G]}\coloneq \kz_G$ for $[G]=([G^{(1)}],\ldots,[G^{(n)}])\in \PP(V)^n$.

\begin{definition}\label{def: Nash locus}
Consider a subvariety $Y\subseteq\PP^\bd$. The {\em Nash locus of $Y$} is
\begin{equation}
    \kn(Y)\coloneqq\overline{\{[G]\in \PP(V)^n\mid \kz_{[G]}\cap Y\neq\emptyset\}}\subseteq \PP(V)^n\,.
\end{equation}
\end{definition}

The previous definition is, of course, completely algebraic, but the motivation comes from Game Theory. Indeed, in the computation of Nash equilibria, there might be scenarios in which the players can choose a strategy by mixing $d_i$ pure strategies, but the way of mixing cannot be arbitrary. In other words, the family of probability distributions allowed is a proper semialgebraic subset of $\prod_{i=1}^{d_i} \Delta_{d_i-1}$. Its Zariski closure in the multiprojective strategy space $\PP^\bd$ is the variety $Y$ written above.

\medskip

In the rest of this section, we present a few examples that motivate Definition~\ref{def: Nash locus}.

\begin{example}\label{ex: really one point}
Assume $Y=\{[\bar{\bpi}]\}\subseteq \PP^\bd$ consists of just one point. If $\bar{\bpi}=(\bar{\pi}_1,\ldots,\bar{\pi}_n)\in\prod_{i=1}^n\Delta_{d_i-1}^\circ$, then $\kn(Y)$ is the variety of all games $[G]$ such that $\bar{\bpi}$ is a totally mixed Nash equilibrium. In general for any $[\bar{\bpi}]\in \PP^\bd$ the Nash locus $\kn(Y)$ is linear of codimension $\sum_{i=1}^n(d_i-1)$. This follows easily from the fact that the entries of $\Delta_i\Pi_i$ are linear in the entries of $G^{(i)}$, and that the variables appearing in entry $i$ and in entry $j\neq i$ are disjoint.
\end{example}

\begin{example}\label{ex: one hyperplane}
Consider $n=2$ players with an equal number of pure strategies $d_1=d_2=d$. This time we fix $Y_1$ to be the hyperplane in $\PP^{d-1}$ of equation $\pi_1^{(1)}-\pi_2^{(1)}=0$ and consider the subvariety $Y=Y_1\times\PP^{d-1}\subseteq\PP^\bd$.
In other words, we are interested in games such that, for some totally mixed Nash equilibrium, the first player is forced to choose a mixed strategy $\pi^{(1)}$ that assigns the same probability to the first two decisions.
In this case $\kn(Y)$ is a hypersurface of bidegree $(0,d-1)$ within $\PP(V)^2$, defined by the equation
\[
\det
\begin{pmatrix}
1 & \cdots & 1\\
g_{21}^{(2)} & \cdots & g_{2d}^{(2)}\\
\vdots & & \vdots\\
g_{d1}^{(2)} & \cdots & g_{dd}^{(2)}
\end{pmatrix}
-
\det
\begin{pmatrix}
g_{11}^{(2)} & \cdots & g_{1d}^{(2)}\\
1 & \cdots & 1\\
\vdots & & \vdots\\
g_{d1}^{(2)} & \cdots & g_{dd}^{(2)}
\end{pmatrix}
=0\,.
\]
The structure of these equations involving determinants is no coincidence. We discuss this phenomenon in more detail in Proposition \ref{prop:eq for 2players}.
\end{example}

The first two examples are ``less'' interesting, meaning that their study reduces to the study of {\em Nash resultants}; see \cite[\S 3.4]{abo2026vector}. We now give an example with a nonlinear subvariety $Y$.

\begin{example}\label{ex: one conic}
Consider $n=2$ players again with an equal number of pure strategies $d_1=d_2=3$. Let $Y_1$ be the conic in $\PP^2$ of equation $2(\pi_2^{(1)})^2-\pi_1^{(1)}\pi_3^{(1)}=0$ and consider $Y=Y_1\times\PP^2$. In this case, we are interested in games such that, for some totally mixed Nash equilibrium, the first player is forced to choose the mixed strategy $\pi^{(1)}$ using a binomial distribution $B(2,1/2)$. The conic $Y_1$ is the Zariski closure of the statistical model associated with $B(2,1/2)$. In this case, $\kn(Y)$ is a hypersurface. Its equation can be derived using the supplementary software \texttt{nashLoci.m2} available at \cite{sodomaco2026supplementary}, via the function \texttt{idealNashLocus}. The unique generator of \texttt{INL} in Figure~\ref{fig: idealNashLocus} is a polynomial of bidegree $(0,4)$ with $54$ terms in the variables of $G^{(2)}$. We investigate this polynomial further in Example~\ref{ex: one conic continued}.
\end{example}
\begin{figure}[h]
\centering
\begin{tcolorbox}[size=fbox,width=.6\linewidth,colback=blue!5!white,colframe=blue!75!black]
\begin{Verbatim}[fontsize=\small,commandchars=\\\{\}]
R = QQ[x_1..x_3]**QQ[y_1..y_3];
xx = matrix\{\{x_1..x_3\}\}; yy = matrix\{\{y_1..y_3\}\};
IY = ideal(2*x_2^2-x_1*x_3);
INL = idealNashLocus(\{3,3\},IY);
\end{Verbatim}
\end{tcolorbox}
\caption{Commands for \texttt{idealNashLocus} given the ideal of $Y$.}\label{fig: idealNashLocus}
\end{figure}

\section{Codimensions and degrees of Nash loci}\label{sec: codim degree}
 
To relate the multidegree of the Nash locus $\kn(Y)$ to that of $Y$, we will work in the Chow rings of $\PP^\bd$, $\PP(V)^n$ and their product $\PP^\bd\times \PP(V)^n$ respectively. For these we write
\begin{align*}
    A^*(\PP^\bd)\ &=\ \Z[h_1,\dots,h_n]/\langle h_1^{d_1},\dots,h_n^{d_n}\rangle\\
    A^*(\PP(V)^n)\ &=\ \Z[t_1,\dots,t_n]/\langle t_1^{d_1\cdots d_n},\dots,t_n^{d_1\cdots d_n}\rangle\\
    A^*(\PP^\bd\times \PP(V)^n)\ &=\ \Z[h_1,\ldots,h_n,t_1,\ldots,t_n]/\langle h_1^{d_1},\ldots, h_n^{d_n},t_1^{d_1\cdots d_n},\ldots,t_n^{d_1\cdots d_n}\rangle\,.
\end{align*}
Here $t_i$ is the pullback of the hyperplane class in the $i$-th factor of $\PP^\bd$ and $h_i$ is the pullback of the hyperplane class in the $i$-th factor of $\PP(V)^n$.

For a subvariety $Y\subseteq\PP^\bd$ we write $[Y]\in A^*(\PP^\bd)$ for its multidegree. When $Y$ is a hypersurface, we also denote the multidegree by an $n$-tuple of numbers rather than a linear polynomial in $n$ variables. In the following statement, given a polynomial $f\in\Z[h_1,\ldots,h_n,t_1,\ldots,t_n]$, the coefficient of the monomial $h^\alpha$ in $f$ means the polynomial $g_\alpha \in \Z[t_1,\ldots,t_n]$ when writing $f=\sum_{\beta}g_\beta(t_1,\ldots,t_n)h^\beta$.

\begin{theorem}\label{thm: codim degree Nash locus}
Assume that $d_1\le\cdots\le d_n$ and that $d_n-1\le\sum_{i=1}^{n-1}(d_i-1)$. Let $Y\subset\PP^\bd$ be an irreducible variety of codimension $c$ and multidegree
\[
    [Y]\ =\ \sum_{|\alpha|=c}\delta_\alpha\,h_1^{\alpha_1}\cdots h_n^{\alpha_n}\in A^*(\PP^\bd)\,.
\]
The Nash locus $\kn(Y)$ is irreducible of the same codimension $c$ in $\PP(V)^{n}$. Its multidegree is the coefficient of the monomial $h_1^{d_1-1}\cdots h_n^{d_n-1}$ in the polynomial
\[
    [Y]\prod_{i=1}^n(\widehat{h}_i+t_i)^{d_i-1}\,,\qquad\widehat{h}_i\ \coloneqq\ \sum_{j\neq i}h_i\,.
\]
\end{theorem}
\begin{proof}
Consider the incidence variety
\begin{equation}\label{eq: incidence variety}
    N \coloneqq \overline{\{([\bpi],[G])\in\PP^\bd\times \PP(V)^n\mid \bpi\in \kz_{[G]}\cap Y\}}
\end{equation}
and denote by $\varphi_1$ and $\varphi_2$ the restrictions to $N$ of the projections onto the two factors of $\PP^\bd\times\PP(V)^n$. On the one hand, the image of $\varphi_1$ is $Y$. On the other hand, the image of $\varphi_2$ is $\kn(Y)$.

Observe that by Example~\ref{ex: really one point} the morphism $\varphi_1$ has linear fibers of constant dimension $\dim\PP(V)^n-\sum_{i=1}^n(d_i-1)=\dim\PP(V)^n-\dim \PP^\bd$ over all points $[\bpi]\in Y$. This makes $N$ a rank $\dim\PP(V)^n-\dim \PP^\bd$ vector bundle over $Y$. Since $Y$ is irreducible, we conclude that $N$ is irreducible of dimension $\dim\PP(V)^n-\dim \PP^\bd +\dim Y=\dim\PP (V)^n-c$. 

We claim that the second morphism $\varphi_2$ is birational. To see this, fix a point $[\bpi]\in Y$. Using the same approach as in \cite[Theorem 4.1]{mckelvey1997maximal} (see also the proof of \cite[Lemma 3.5]{abo2026vector}), it is possible to find a game $G\in V^{\oplus n}$ such that $\kz_G$ is the zero scheme of a global section $\sigma$ of the vector bundle $E$ in \eqref{eq: def vector bundle E} with the following properties: all components $\sigma_k^{(i)}$ of $\sigma$ are products of linear polynomials in the $n$ sets of variables of $\PP^\bd$; $\kz_G$ is zero-dimensional and reduced of the maximal cardinality $c(\bd)$, and finally, exactly one of these $c(\bd)$ points lies on $Y$. The last two properties use the assumption $d_n-1\le\sum_{i=1}^{n-1}(d_i-1)$. 
Therefore $\kn(Y)$ is irreducible, as the image of an irreducible variety, and by finiteness of $\varphi_2$ we get $\dim \kn(Y)=\dim N=\dim \PP(V)^n-c$. 

Notice that the incidence variety $N$ is the transverse intersection of $Y$ with all the equations $\sigma_k^{(i)}$ where $i\in[n]$ and $k\in[d_i-1]$. Since $\sigma_k^{(i)}$ is linear in the variables appearing in $G^{(i)}$ and linear in the variables in $\pi^{(j)}$ for $j\neq i$, we get that
\begin{align*}
    [N]=[Y]\prod_{i=1}^n\left(t_i+\sum_{j\neq i}h_j\right)^{d_i-1} =\ [Y]\prod_{i=1}^n(\widehat{h}_i+t_i)^{d_i-1}\in A^*(\PP^\bd \times \PP(V)^n)\,.
\end{align*}
The multidegree of $\kn(Y)$ is a homogeneous polynomial in the variables $t_1,\ldots,t_n$ of degree $c=\codim_{\PP(V)^n}\kn(Y)$. Since the map $\varphi_2$ is birational, it follows from the push-pull formula that the coefficient in front of a monomial $t^\alpha$ equals the coefficient in front of the monomial $h_1^{d_1-1}\cdots h_n^{d_n-1}t^\alpha$ in $[\kn(Y)]$. 
\end{proof}

Note that the previous formula coincides with \cite[Theorem 2.8]{abo2026vector} when $Y=\PP^\bd$, namely when no constraint is imposed. For two players and arbitrary $Y$ this specializes to the following formula.

\begin{cor}\label{cor: multidegree two factors}
Let $n=2$ and $Y\subseteq\PP^\bd=\PP^{d_1-1}\times\PP^{d_2-1}$ of codimension $c$ and multidegree
\[
    [Y] = \sum_{k=0}^c b_k\,h_1^kh_2^{c-k}\,,
\]
where we set $b_k=0$ whenever $k\geq d_1$ or $r-k\geq d_2$. Then (assuming $\binom{a}{b}=0$ for $b>a$)
\[
\codim\kn(Y)\ =\ c\quad\text{and}\quad[\kn(Y)]\ =\ \sum_{k=0}^c\binom{d_1-1}{c-k}\binom{d_2-1}{k}b_k\,t_1^kt_2^{c-k}\,.
\]
\end{cor}

\begin{example}
Let $d_1=d_2=d$ and let $Y=\Delta\subseteq\PP^{d-1}\times\PP^{d-1}$ be the diagonal. Then $[Y]=\sum_{k=0}^{d-1}h_1^kh_2^{d-1-k}$ and hence by Corollary~\ref{cor: multidegree two factors}
\[
    [\kn(Y)]\ =\ \sum_{k=0}^{d-1}\binom{d-1}{k}\binom{d-1}{d-1-k}t_1^kt_2^{d-1-k}\ =\ \sum_{k=0}^{d-1}\binom{d-1}{k}^2t_1^kt_2^{d-1-k}\,.
\]
The following identity can also be verified via the function {\tt degNashLocus} shown in Figure~\ref{fig: degNashLocus diagonal}, using either the ideal of $Y$ (on the left) or the multidegree of $Y$ (on the right, after defining the correct quotient ring) as input.
\end{example}

\begin{figure}[h]
\centering
\begin{minipage}{0.49\textwidth}
\begin{tcolorbox}[size=fbox,width=\linewidth,colback=blue!5!white,colframe=blue!75!black]
\begin{Verbatim}[fontsize=\scriptsize,commandchars=\\\{\}]
n = 2; d = 3;
R = tensor apply(n, i-> QQ[x_(i+1,1)..x_(i+1,d)]);
Delta = minors(2, genericMatrix(R,d,n));
dNL = degNashLocus(toList(n:d),Delta);
A = ring dNL;
B = sum(d, i->(binomial(d-1,i))^2*t_1^i*t_2^(d-1-i));
dNL == B
\end{Verbatim}
\end{tcolorbox}
\end{minipage}
\begin{minipage}{0.49\textwidth}
\begin{tcolorbox}[size=fbox,width=\linewidth,colback=blue!5!white,colframe=blue!75!black]
\begin{Verbatim}[fontsize=\scriptsize,commandchars=\\\{\}]
n = 2; d = 3;
S = ZZ[h_1..h_n]/ideal(apply(n, i-> h_(i+1)^d));
mdegDelta = sum first entries basis((n-1)*(d-1), S);
dNL = degNashLocus(toList(n:d),mdegDelta);
A = ring dNL;
B = sum(d, i->(binomial(d-1,i))^2*t_1^i*t_2^(d-1-i));
dNL == B
\end{Verbatim}
\end{tcolorbox}
\end{minipage}
\caption{Computation of $[\kn(Y)]$ for the diagonal $Y=\Delta\subseteq\PP^2\times\PP^2$. The code can be used for any diagonal in a product of $n$ projective spaces of the same dimension $d-1$.}\label{fig: degNashLocus diagonal}
\end{figure}

\begin{example}\label{ex: Y splits}
A natural case is when $Y$ is itself a product of subvarieties, namely when the constraints are not shared among the players.
In this case, one might derive a formula for the codimension and the degree of the Nash locus $\kn(Y)$ as a function of the degrees and the dimensions of the factors of $Y$.
Assume that $d_1\le\cdots\le d_n$ and that $d_n-1\le\sum_{i=1}^{n-1}(d_i-1)$. For every $i\in[n]$, let $Y_i\subset\PP^{d_i}$
be an irreducible projective variety of codimension $\
c_i$. Let $Y=\prod_{i=1}^n Y_i\subset\PP^\bd$.
The Nash locus $\kn(Y)$ is irreducible of codimension $c=\sum_{i=1}^nc_i$ in $\PP(V)^{n}$. Its multidegree is $\delta\cdot\prod_{i=1}^n\deg Y_i$, where $\delta$ is the coefficient of the monomial $h_1^{d_1-1}\cdots h_n^{d_n-1}$ in the polynomial $h_1^{c_1}\cdots h_n^{c_n}\prod_{i=1}^{n}(\widehat{h_i}+t_i)^{d_i-1}$, where $\widehat{h_i}\coloneqq\sum_{j\neq i}h_j$.
\end{example}

\begin{exampletwo}
We apply Theorem \ref{thm: codim degree Nash locus} to the  examples of Section \ref{sec: definition}:
\begin{enumerate}
    \item In Example \ref{ex: really one point}, we have $[Y]=h_1^{d_1-1}\cdots h_n^{d_n-1}$, hence $c=\prod_{i=1}^n(d_i-1)$ and $[\kn(Y)]=t_1^{d_1-1}\cdots t_n^{d_n-1}$.
    This confirms that $\kn(Y)$ is a linear subspace. See the script on the left-hand side of Figure~\ref{fig: degNashLocus}.
    \item In Example \ref{ex: one hyperplane}, we are in the setting of Example \ref{ex: Y splits} for $c_1=1$ and $c_2=0$, hence the codimension is $c=1$ and $[\kn(Y)]=(d-1)\,t_2$ is the coefficient of $h_1^{d-1}h_2^{d-1}$ in the polynomial $h_1(h_2+t_1)^{d-1}(h_1+t_2)^{d-1}=h_1(h_1h_2+h_1t_1+h_2t_2+t_1t_2)^{d-1}$. This confirms that $\kn(Y)$ is a hypersurface of bidegree $(0,d-1)$.
    \item In Example \ref{ex: one conic}, we are again in the setting of Example \ref{ex: Y splits} for $c_1=1$ and $c_2=0$, hence the codimension is $c=1$ and $[\kn(Y)]=4\,t_2$ is the coefficient of $h_1^2h_2^2$ in the polynomial $2\,h_1(h_2+t_1)^2(h_1+t_2)^2=2\,h_1(h_1h_2+h_1t_1+h_2t_2+t_1t_2)^2$. This confirms that $\kn(Y)$ is a hypersurface of bidegree $(0,4)$. For this and the previous case, see the script on the right-hand side of Figure~\ref{fig: degNashLocus}.\hfill$\diamond$
\end{enumerate}
\end{exampletwo}

\begin{figure}[h]
\centering
\begin{minipage}{0.48\textwidth}
\begin{tcolorbox}[size=fbox,width=\linewidth,colback=blue!5!white,colframe=blue!75!black]
\begin{Verbatim}[fontsize=\small,commandchars=\\\{\}]
D = \{2,3,4\};
R = ZZ[h_1..h_(#D)]/
(ideal apply(#D, i-> h_(i+1)^(D#i)));
y = product(#D, i-> h_(i+1)^(D#i-1));
degNashLocus(D,y)
\end{Verbatim}
\end{tcolorbox}
\end{minipage}
\begin{minipage}{0.48\textwidth}
\begin{tcolorbox}[size=fbox,width=\linewidth,colback=blue!5!white,colframe=blue!75!black]
\begin{Verbatim}[fontsize=\small,commandchars=\\\{\}]
D = \{3,3\};
R = ZZ[h_1..h_(#D)]/
(ideal apply(#D, i-> h_(i+1)^(D#i)))
y = 2*h_1
degNashLocus(D,y)
\end{Verbatim}
\end{tcolorbox}
\end{minipage}
\caption{Commands for {\tt degNashLocus} given the multidegree of $Y$ as input.}\label{fig: degNashLocus}
\end{figure}

In what follows, we focus on two-player games $G=(G^{(1)},G^{(2)})$ with equal numbers of pure strategies $d=d_1=d_2$. Given a matrix $A$, we denote by $A_i$ the matrix obtained by replacing the $i$-th column of $A$ with a vector of ones. If $A$ is a nonsingular square matrix of size $d$ and $x=(x_1,\ldots,x_d)^\mT$, then $\rank\begin{pmatrix}Ax &\bone\end{pmatrix}\le 1$ if and only if $x^\mT$ is proportional to $\bar{x}\coloneqq(\det(A_1),\ldots,\det(A_d))$ by Cramer's rule.

\begin{proposition}\label{prop:eq for 2players}
Let $n=2$ and $d\coloneqq d_1=d_2$. Let $Y \subseteq \PP^{d-1}\times \PP^{d-1}$ be an irreducible subvariety whose prime ideal is $\langle f_1,\ldots,f_s\rangle\subseteq R=R_1\otimes R_2$. The Nash locus $\kn(Y)$ is set-theoretically cut out by 
\[
    \left\langle f_1(\bar{\pi}^{(1)},\bar{\pi}^{(2)}),\ldots,f_s(\bar{\pi}^{(1)},\bar{\pi}^{(2)})\right\rangle:\left(\langle \bar{\pi}^{(1)}_1,\dots,\bar{\pi}^{(1)}_d\rangle \cap \langle\bar{\pi}^{(2)}_1,\dots,\bar{\pi}^{(2)}_d\rangle\right)^\infty\subseteq S=S_1\otimes S_2,
\]
where $\bar{\pi}^{(1)} = (\det(G_1^{(2)}),\ldots,\det(G_d^{(2)}))$ and $\bar{\pi}^{(2)} = (\det(G_1^{(1)}),\ldots,\det(G_d^{(1)}))$.

Those games that admit a totally mixed Nash equilibrium on $Y$ are obtained from the real points of $\kn(Y)$ by imposing the additional semialgebraic conditions $\det(G_i^{(1)})>0$ and $\det(G_i^{(2)})>0$ for all $i\in[d]$.
\end{proposition}
\begin{proof}
Given a game $G=(G^{(1)},G^{(2)})$, the condition of a point $[\bpi]=([\pi^{(1)}],[\pi^{(2)}])\in\PP^{d-1}\times \PP^{d-1}$ to be in the Nash equilibrium scheme $\kz_{[G]}$ is by definition
\[
    \rank\begin{pmatrix}G^{(2)}\pi^{(1)} &\bone\end{pmatrix}\le 1\,,\quad\quad\rank\begin{pmatrix}G^{(1)}\pi^{(2)} &\bone\end{pmatrix}\le 1\,.
\]
Assuming that $G^{(1)}$ and $G^{(2)}$ both have full rank, by Cramer's rule this is equivalent to $(\pi^{(1)},\pi^{(2)})=(\bar{\pi}^{(1)},\bar{\pi}^{(2)})$. Hence such full rank games belong to $\kn(Y)$ precisely when the defining equations of $Y$ vanish at $(\bar{\pi}^{(1)},\bar{\pi}^{(2)})$ and $(\bar{\pi}^{(1)},\bar{\pi}^{(2)})$ is a well-defined point in biprojective space, i.e. not all entries of either component vanish. A generic game $[G]\in\kn(Y)$ has a zero-dimensional, reduced Nash equilibrium scheme $\kz_{[G]}$ of cardinality 1. This can be verified by the same argument used in the proof of Theorem \ref{thm: codim degree Nash locus}. Therefore $\kn(Y)$ is contained in the variety cut out by the saturation of $\langle f_1(\bar{\pi}^{(1)},\bar{\pi}^{(2)}),\ldots,f_s(\bar{\pi}^{(1)},\bar{\pi}^{(2)})\rangle$ in $\langle \bar{\pi}^{(1)}_1,\dots,\bar{\pi}^{(1)}_d\rangle \cap \langle\bar{\pi}^{(2)}_1,\dots,\bar{\pi}^{(2)}_d\rangle$. For the other inclusion, simply note that if some $G$ is in the variety cut out by these equations, then $(\bar{\pi}^{(1)},\bar{\pi}^{(2)})$ belongs to $\kz_{[G]}\cap Y$.
\end{proof}

\begin{example}\label{ex: one conic continued}
We come back to Example \ref{ex: one conic} where $d=d_1=d_2=3$, $Y_1\subseteq\PP^2$ is the conic of equation $(\pi_2^{(1)})^2-\pi_1^{(1)}\pi_3^{(1)}=0$, and $Y=Y_1\times\PP^2$. Then $\kn(Y)$ is cut out by the polynomial
\begin{align*}
2\det\begin{pmatrix}
g_{1,1}^{(2)} & g_{1,2}^{(2)} & g_{1,3}^{(2)}\\
1 & 1 & 1 \\
g_{3,1}^{(2)} & g_{3,2}^{(2)} & g_{3,3}^{(2)}
\end{pmatrix}^2+
\det\begin{pmatrix}
1 & 1 & 1 \\
g_{2,1}^{(2)} & g_{2,2}^{(2)} & g_{2,3}^{(2)}\\
g_{3,1}^{(2)} & g_{3,2}^{(2)} & g_{3,3}^{(2)}
\end{pmatrix}\det
\begin{pmatrix}
g_{1,1}^{(2)} & g_{1,2}^{(2)} & g_{1,3}^{(2)}\\
g_{2,1}^{(2)} & g_{2,2}^{(2)} & g_{2,3}^{(2)}\\
1 & 1 & 1
\end{pmatrix}
\end{align*}
which is the polynomial with 54 terms computed in Example \ref{ex: one conic}. The saturation step can be skipped here. The previous identity can be verified using the script in Figure~\ref{fig: idealNashLocus2}, after running the lines in Figure~\ref{fig: idealNashLocus}.
\end{example}
\begin{figure}[h]
\centering
\begin{tcolorbox}[size=fbox,width=\linewidth,colback=blue!5!white,colframe=blue!75!black]
\begin{Verbatim}[fontsize=\small,commandchars=\\\{\}]
G1 = matrix for i in 1..3 list for j in 1..3 list if i==1 then 1 else G_(2,i,j);
G2 = matrix for i in 1..3 list for j in 1..3 list if i==2 then 1 else G_(2,i,j);
G3 = matrix for i in 1..3 list for j in 1..3 list if i==3 then 1 else G_(2,i,j);
INL == sub(ideal(2*det(G2)^2-det(G1)*det(G3)), ring INL)
\end{Verbatim}
\end{tcolorbox}
\caption{The determinantal check of Example~\ref{ex: one conic continued}.}\label{fig: idealNashLocus2}
\end{figure}

\begin{example}
Let $Y\subseteq \PP^2\times \PP^2$ be the point $([1:1:1],[1:1:1])$. Then $Y$ is cut out by the polynomials $\pi_1^{(1)}-\pi_2^{(1)},\pi_1^{(1)}-\pi_3^{(1)},\pi_1^{(2)}-\pi_2^{(2)},\pi_1^{(2)}-\pi_3^{(2)}$. In this case $\bar{\pi}^{(1)} = (\det(G_1^{(2)}),\det(G_2^{(2)}),\det(G_3^{(2)}))$ and $\bar{\pi}^{(2)} = (\det(G_1^{(1)}),\det(G_2^{(1)}),\det(G_3^{(1)}))$. When substituting $(\pi^{(1)},\pi^{(2)})=(\bar{\pi}^{(1)},\bar{\pi}^{(2)})$ into the four polynomials defining $Y$, we get a non-prime ideal which decomposes into four prime ideals, all of codimension 4:
\begin{itemize}
    \item[$(i)$] One linear ideal cutting out $\kn(Y)$, i.e. games for which $[\bar{\pi}^{(1)}]=[1:1:1]$ and $[\bar{\pi}^{(2)}]=[1:1:1]$.
    \item[$(ii)$] Two ideals of degree 3 which correspond to games for which either $\bar{\pi}^{(1)}=0$ and $\bar{\pi}^{(2)}=[1:1:1]$ or $\bar{\pi}^{(1)}=[1:1:1]$ and $\bar{\pi}^{(2)}=0$ respectively.
    \item[$(iii)$] One ideal of degree 9 which corresponds to games for which $\bar{\pi}^{(1)}=\bar{\pi}^{(2)}=0$.
\end{itemize}
The saturation is necessary here and precisely removes the three unwanted components.
\end{example}

\begin{example}
If $d_1=d_2=d$ and $Y\subseteq \PP^{d-1}\times \PP^{d-1}$ is the diagonal, then the defining equations of $\kn(Y)$ are the $d\times d$-minors of the $2(d-1)\times d$ matrix
\begin{equation}\label{eq: def matrix D}
    D=
    \begin{pmatrix}
        g_{1,1}^{(2)}-g_{1,2}^{(2)} & \cdots & g_{1,1}^{(2)}-g_{1,d}^{(2)} & g_{1,1}^{(1)}-g_{2,1}^{(1)} & \cdots & g_{1,1}^{(1)}-g_{d,1}^{(1)} \\
        \vdots & & \vdots & \vdots & & \vdots \\
        g_{d,1}^{(2)}-g_{d,2}^{(2)} & \cdots & g_{d,1}^{(2)}-g_{d,d}^{(2)} & g_{1,d}^{(1)}-g_{2,d}^{(1)} & \cdots & g_{1,d}^{(1)}-g_{d,d}^{(1)}
    \end{pmatrix}^{\!\mT}\,.
\end{equation}
Indeed a game $G=(G^{(1)},G^{(2)})$ belongs to $\kn(Y)$ if and only if there exists a point $[\pi]\in \PP^{d-1}$ such that both $\rank\begin{pmatrix}G^{(2)}\pi &\bone\end{pmatrix}\le 1$ and $\rank\begin{pmatrix}G^{(1)}\pi &\bone\end{pmatrix}\le 1$. By taking $2\times 2$-minors this is equivalent to
\begin{align*}
    g_{1,1}^{(j)}\pi_1+\cdots+g_{d,1}^{(j)}\pi_d\ =\ g_{1,i}^{(j)}\pi_1+\cdots+g_{d,i}^{(j)}\pi_d
\end{align*}
for all $i\in\{2,\dots,d\}$ and $j\in\{1,2\}$. Equivalently, $D\,\pi=0$ and so $G$ belongs to $\kn(Y)$ if and only if $D$ has a nonzero kernel vector. Going back to the Rock, Paper and Scissors game of Table~\ref{tab: RPS}, after plugging in the payoffs into the matrix $D$, we obtain the $4\times 3$ matrix
\[
    \begin{pmatrix}
        -1 & 1 & -1 & 1\\
        -1 & -2 & -1 & -2\\
        2 & 1 & 2 & 1
    \end{pmatrix}^{\!\mT}
\]
which admits the nonzero kernel $\langle(1,1,1)\rangle$, corresponding to the unique totally mixed Nash equilibrium given in \eqref{eq: tmNE RPS}. This confirms that Rock, Paper and Scissors belongs to $\kn(Y)$ for the diagonal $Y\subseteq \PP^2\times\PP^2$.
\end{example}

We conclude this section with the computation of the ideal of a Nash locus for the binary three-player format.

\begin{example}
Let $d_1=d_2=d_3=2$ and consider the subvariety $Y=\{[0:1]\}\times\PP^1\times\PP^1\subseteq(\PP^1)^3$. In other words, we are looking for three-player binary games with the property that at least one solution $[\bpi]=([\pi_1^{(1)},\pi_2^{(1)}],[\pi_1^{(2)},\pi_2^{(2)}],[\pi_1^{(3)},\pi_2^{(3)}])$ of $\kz_G$ is such that $\pi_1^{(1)}=0$.
The equations defining $\kz_G$ for a three-player binary game $G=(G^{(1)},G^{(2)},G^{(3)})$ are
\begin{align}
    \begin{split}
        a_{11}\pi_1^{(2)}\pi_1^{(3)}+a_{12}\pi_1^{(2)}\pi_2^{(3)}+a_{21}\pi_2^{(2)}\pi_1^{(3)}+a_{22}\pi_2^{(2)}\pi_2^{(3)} &= 0\,,\\
        b_{11}\pi_1^{(1)}\pi_1^{(3)}+b_{12}\pi_1^{(1)}\pi_2^{(3)}+b_{21}\pi_2^{(1)}\pi_1^{(3)}+b_{22}\pi_2^{(1)}\pi_2^{(3)} &= 0\,,\\
        c_{11}\pi_1^{(1)}\pi_1^{(2)}+c_{12}\pi_1^{(1)}\pi_2^{(2)}+c_{21}\pi_2^{(1)}\pi_1^{(2)}+c_{22}\pi_2^{(1)}\pi_2^{(2)} &= 0\,,
    \end{split}
\end{align}
where $a_{ij}=g_{1ij}^{(1)}-g_{2ij}^{(1)}$, $b_{ij}=g_{i1j}^{(2)}-g_{i2j}^{(2)}$ and $c_{ij}=g_{ij1}^{(3)}-g_{ij2}^{(3)}$ for all $i,j$ in $\{1,2\}$. Substituting $[\pi_1^{(1)}:\pi_2^{(1)}]=[0:1]$ in the second and third equation, one gets $[\pi_1^{(2)}:\pi_2^{(2)}]=[c_{22}:-c_{21}]$ and $[\pi_1^{(3)}:\pi_2^{(3)}]=[b_{22}:-b_{21}]$. Substituting these relations into the first equation, one obtains the equation of the Nash locus $\kn(Y)$:
\[
    a_{11}b_{22}c_{22}-a_{12}b_{21}c_{22}-a_{21}b_{22}c_{21}+a_{22}b_{21}c_{21}=0\,.
\]
This corresponds to a hypersurface in $(\C^{2\times 2\times 2})^{\oplus 3}$ of tridegree $(1,1,1)$ in the three sets of variables of the payoff tensors of $G$.

\medskip

In the remainder of the example, we exhibit two games which belong to $\kn(Y)$, and with different types of Nash equilibria. In the supplementary material \cite{sodomaco2026supplementary}, we provide the \texttt{Macaulay2} code for these examples, which uses the \texttt{GameTheory} package \cite{GameTheorySource}. First, consider the three-player binary game $G_1$ with payoff table
\[
    \begin{bmatrix}[cc:cc]
        (0,1,1) & (1,0,0) & (1,-1,1) & (0,-3,0)\\[2pt]
        (0,-1,1) & (0,-1,-1) & (1,1,0) & (2,1,2)
    \end{bmatrix}\,,
\]
where, in our convention, player $1$ chooses a row, player $2$ chooses a column, and player $3$ chooses a side of the table (the two sides are separated by a dashed line). The Nash equilibrium scheme of $\kz_{G_1}$ is the zero locus of the system
\[
(\pi_1^{(3)}-2\,\pi_2^{(3)})\,\pi_2^{(2)} = (\pi_1^{(3)}+2\,\pi_2^{(3)})\,\pi_1^{(1)} = \pi_2^{(1)}\,(\pi_1^{(2)}-3\,\pi_2^{(2)}) = 0\,,
\]
which consists of the two reduced points $([0:1],[3:1],[2:1])$ and $([1:0],[1:0],[2:-1])$. Neither of these two points corresponds to a totally mixed Nash equilibrium, while the first one is a mixed Nash equilibrium of $G_1$: indeed, using the definition \eqref{eq: opt NE}, we have
\begin{align*}
    \Pi_1\left(\pi^{(1)},\left(\frac{3}{4},\frac{1}{4}\right),\left(\frac{2}{3},\frac{1}{3}\right)\right) &= \frac{5}{12}\left(\pi_1^{(1)}+\pi_2^{(1)}\right) = \frac{5}{12}\quad\forall\,\pi^{(1)}\in\Delta_1\\
    \Pi_2\left((0,1),\pi^{(2)},\left(\frac{2}{3},\frac{1}{3}\right)\right) &= -\frac{1}{3}\left(\pi_1^{(2)}+\pi_2^{(2)}\right) = -\frac{1}{3}\quad\forall\,\pi^{(2)}\in\Delta_1\\
    \Pi_3\left((0,1),\left(\frac{3}{4},\frac{1}{4}\right),\pi^{(1)}\right) &= \frac{1}{2}\left(\pi_1^{(3)}+\pi_2^{(3)}\right) = \frac{1}{2}\quad\forall\,\pi^{(3)}\in\Delta_1\,.
\end{align*}
The second example is ``Selten's Horse'', studied in \cite[Section 4.3]{jahani2022automated}, a three-player binary game $G_2$ with payoff table
\[
    \begin{bmatrix}[cc:cc]
        (3,2,2) & (3,2,2) & (0,0,0) & (0,0,0)\\[2pt]
        (4,4,0) & (1,1,1) & (0,0,1) & (1,1,1)
    \end{bmatrix}\,.
\]
It was already discussed in \cite[Example 3.13]{abo2026vector} that $\kz_{G_2}$ is zero-dimensional of degree $2$ and nonreduced; in particular, it consists of the double point $[\bpi]=([0:1],[1:-1],[1:3])\in(\PP^1)^3$. This means that $G_2$ belongs to $\kn(Y)$ but also to the nonsingular locus of the {\em Nash discriminant variety} $\Delta(2,2,2)$; see \cite[Definition 3.2]{abo2026vector}. Consequently, we have that a generic point in the intersection between $\Delta(2,2,2)$ and $\kn(Y)$ is nonsingular in $\Delta(2,2,2)$. Note that, in the {\tt Macaulay2} code used, the function {\tt nashEquilibriumIdeal} does not compute the ideal of $\kz_{[G]}$. In this example, the function automatically includes the relation $\pi_1^{(1)}+\pi_2^{(1)}-1$, which is not satisfied by the point $[1:-1]$.
\end{example}

\section{Nash loci and multigraded associated varieties}\label{sec: multigraded associated varieties}

The goal of this section is to connect Nash loci with incidence loci in products of Grassmannians. This will lead to the study of {\em multigraded associated varieties}, generalizing \cite{osserman2019multigraded} from one-codimensional loci to incidence loci of arbitrary codimension.

\medskip

Fix a game format $\bd=(d_1,\dots,d_n)$ and a subvariety $Y\subseteq \PP^\bd$. By definition, the Nash locus $\kn(Y)$ is the variety of all games $G$ whose Nash equilibrium scheme intersects $Y$. For a fixed game $G$, we observe that the defining equations of the Nash equilibrium scheme are linear in the monomials
\[
    p^{(i)}_{j_1\cdots\,\hat{j_i}\,\cdots j_n}\ \coloneqq\ \pi_{j_1}^{(1)}\cdots \pi_{j_{i-1}}^{(i-1)}\pi_{j_{i+1}}^{(i+1)}\cdots \pi_{j_n}^{(n)}
\]
where $i\in [n]$ and $j_k\in [d_k]$ for $k\neq i$ vary. Furthermore, all monomials $p^{(i)}_{j_1\cdots\,\hat{j_i}\,\cdots j_n}$ appearing in a single equation share the same index $i\in [n]$. It is therefore natural to consider the mixed Segre embedding $\sigma_\bd\colon\PP^\bd\hookrightarrow\PP=\PP^{k_1}\times\cdots\times\PP^{k_n}$, where $k_i=-1+\prod_{j\neq i}d_j$, defined by
\begin{equation}\label{eq: segre}
    [\bpi]=([\pi^{(1)}],\ldots,[\pi^{(n)}])\longmapsto\left([\pi^{(1)}\otimes\cdots\otimes\pi^{(i-1)}\otimes\pi^{(i+1)}\otimes\cdots\otimes\pi^{(n)}]\right)_{i=1}^n\,,
\end{equation}
which was already described in \cite[\S 5]{pratt2026multigraded}.
By the above discussion, the game $G$ then determines uniquely a system of linear equations on each $\PP^{k_i}$ whose solution set is a product of linear spaces $L_G=L_1\times \cdots \times L_n$ with the property that $\pi\in \PP^d$ belongs to the Nash equilibrium scheme $\kz_G$ if and only if $\sigma_\bd(\pi)\in L_G$. On the other hand, given a product of linear spaces $L=L_1\times \cdots \times L_n$, assuming $\dim(L_i)\geq k_i-d_i+1$, it is easy to see that there always exists (non-uniquely) a game $G$ such that $L_G=L$. In this sense, the study of $\kn(Y)$ reduces to studying the locus of all products of linear spaces $L=L_1\times \cdots \times L_n$ with $\dim(L_i)\geq k_i-d_i+1$ such that $L$ intersects the mixed Segre embedding of $Y$, i.e., $L\cap \sigma_\bd(Y)\neq \emptyset$. We therefore wish to study {\em multigraded associated varieties}, i.e., varieties of products of lines which intersect a given multiprojective variety. The purpose of this section is to define these varieties formally and to study first properties such as dimension and multidegree.

\medskip

First, we recall the definition of associated variety in a single projective space. Let $X\subseteq\PP^{k_1}=\PP^k$ be a projective variety of codimension $c$, and let $\G(\ell,k)$ be the Grassmannian of $\ell$-planes in $\PP^k$. In the following, we use $\G_\ell$ as a shorthand for $\G(\ell,k)$.

\begin{definition}[{\cite[Definition 4.10]{kohn2018isotropic}}]\label{def: associated variety}
Let $0\le \ell \le k-1$. The {\em $\ell$-th associated variety of $X\subseteq\PP^k$} is
\[
    \mathrm{Ch}_\ell(X)\ \coloneqq\ \left\{L\in \G_\ell\mid L\cap X\neq \emptyset\right\}\,.
\]
\end{definition}
The associated varieties of Definition~\ref{def: associated variety} are called {\em varieties of incidence planes} in \cite[Examples 6.14, 11.18, 16.6]{harris1992algebraic}.
When $\ell=c-1$, the locus $\mathrm{Ch}_{c-1}(X)$ is the classical \emph{Chow hypersurface}, i.e., the locus of $(c-1)$-dimensional planes meeting $X$; see \cite{waerdenvander1937geometrie}.
For smaller $\ell$, the locus $\mathrm{Ch}_\ell(X)$ typically has higher codimension in $\G_\ell$; in particular, it is not cut out by a single equation in Pl\"ucker coordinates.

\medskip

If we define the incidence correspondence $I_\ell(X) \ \coloneqq\ \{(x,L)\in X\times \G_\ell \mid x\in L\}$ and let $\pi_2\colon I_\ell\to \G_\ell$ be the second projection, then $\mathrm{Ch}_\ell(X) = \pi_2(I_\ell(X))$. Consider the tautological exact sequence of vector bundles on $\G_\ell$
\begin{equation}\label{eq: tautological sequence}
    0\to S \longrightarrow \ko_{\G_\ell}^{\oplus (k+1)} \longrightarrow Q \to 0\,,
\end{equation}
where $\rank(S)=\ell+1$ and $\rank(Q)=k-\ell$.
The Chow ring $A^*(\G_\ell)$ is a free abelian group generated by the Schubert cycles \cite[Corollary 4.7]{EH}. It follows moreover from Giambelli's formula \cite[Proposition 4.16]{EH} that it is generated multiplicatively by just the {\em special Schubert cycles} $\sigma_j$, which are the Chern classes $c_j(Q)$ of the universal quotientbundle $Q$.

The incidence variety $I_\ell$ is canonically isomorphic to the projective bundle $\PP(S)$ with projection $\pi\colon \PP(S)\to\G_\ell$, where $\PP(S)$ parametrizes pairs $(W,[v])$ with $[v]\in \PP(W)\subseteq \PP^k$. Let $\zeta\coloneqq c_1(\ko_{\PP(S)}(1))$ and let $H\coloneqq c_1(\ko_{\PP^k}(1))$ be the hyperplane class on $\PP^k$.
In the following proposition, we relate the class $[\mathrm{Ch}_\ell(X)]$ in the Chow ring of $\G_\ell$ to the degree of $X$ and the special Schubert classes.

\begin{proposition}\label{prop:class-single}
Assume $X\subseteq \PP^k$ is equidimensional of codimension $c$. Then
\[
    [\mathrm{Ch}_\ell(X)] = \deg(X)\,\sigma_{c-\ell}\in A^*(\G_\ell)\quad\text{for all $\ell\le c$}\,.
\]
\end{proposition}
\begin{proof}
Consider $\PP(S)\cong I_\ell \subseteq \PP^k\times\G_\ell$ with projections $p\colon\PP(S)\to \PP^k$ and $\pi\colon\PP(S)\to\G_\ell$.
One has $p^*H=\zeta$ (see \cite[Section 9.3.1]{EH}).
Pulling back $[X]=\deg(X)\,H^c\in A^c(\PP^k)$ gives $p^*[X] = \deg(X)\,\zeta^c\in A^c(\PP(S))$. Pushing forward yields
\[
    [\mathrm{Ch}_\ell(X)] \ =\  \pi_*\, p^*[X] \ =\  \deg(X)\,\pi_*\bigl(\zeta^c\bigr)\,.
\]
Since $\pi\colon\PP(S)\to\G_\ell$ is a projective bundle of relative dimension $\ell$, the Segre class pushforward formula (see \cite[\S3.3 and \S14]{fulton1998intersection}) yields $\pi_*\bigl(\zeta^{\,\ell+i}\bigr)= s_i(S)$ for all $i\ge 0$, where $s_i(S)$ is the $i$-th Segre class of $S$.
In particular $\pi_*\bigl(\zeta^c\bigr)= s_{c-\ell}(S)$.
Finally, using $s(S)=c(S)^{-1}$ and the Whitney sum formula $c(S)\,c(Q)=1$ coming from the tautological sequence \eqref{eq: tautological sequence}, we get $s_{c-\ell}(S)=c_{c-\ell}(Q)$. Thus $[\mathrm{Ch}_\ell(X)]=\deg(X)\,\sigma_{c-\ell}$.
\end{proof}

Let $\iota\colon\G_l\hookrightarrow \PP(\bigwedge^{\ell+1}\C^{k+1})$ be the Pl\"ucker embedding.
Let $H_{\G_l}\in A^*(\G_l)$ be the Pl\"ucker hyperplane class; equivalently, $H_{\G_l} = c_1\bigl(\det(Q)\bigr)=\sigma_1$. The Pl\"ucker degree of a subvariety $Y\subseteq\G_\ell$ of dimension $m$ is
\[
    \deg_{\mathrm{Pl}}(Y) \ \coloneqq\  \int_{\G_\ell} [Y]\cdot H_{\G_\ell}^{\,m}\,.
\]
Set $N_\ell\coloneqq\dim\G_\ell=c\,(\ell+1)$.
Then, using Proposition~\ref{prop:class-single}, the Pl\"ucker degree of $\mathrm{Ch}_\ell(X)$ is
\[
    \deg_{\mathrm{Pl}}(\mathrm{Ch}_\ell(X)) \ =\  \deg(X)\cdot \int_{\G_\ell} \sigma_{c-\ell}\, H_{\G_\ell}^{\,N_\ell-c+\ell}\,.
\]

\bigskip

We now move to the multigraded setting. Consider the multiprojective space $\PP \coloneqq \PP^{k_1}\times\cdots\times \PP^{k_n}$ and let $X\subseteq\PP$ be an irreducible subvariety of codimension $c\coloneqq\codim_{\PP}(X)$.
For any $\alpha=(\alpha_1,\dots,\alpha_n)$ with $0\le \alpha_i\le k_i$, we set $|\alpha|\coloneqq\sum_{i=1}^n\alpha_i$. We also define $\G_i\coloneqq\G(\alpha_i,k_i)$ for all $i\in[n]$ and $\G_\alpha\coloneqq\G_1\times\cdots\times\G_n$.
In particular, an element of $\G_\alpha$ corresponds to a product of linear spaces $L=L_1\times\cdots\times L_n\subseteq \PP$, where $\dim L_i=\alpha_i$.

\begin{definition}\label{def: multigraded associated variety}
The {\em multigraded associated variety of $X\subseteq\PP$ of index $\alpha$} is
\[
    \mathrm{Ch}_\alpha(X)\ \coloneqq\ \left\{L\in \G_\alpha \mid L\cap X\neq\emptyset\right\}\,.
\]
\end{definition}

When $|\alpha|=c-1$, under suitable nondegeneracy hypotheses, the locus $\mathrm{Ch}_\alpha(X)$ is expected to be a hypersurface in $\G_\alpha$; its defining multihomogeneous equation is the {\em multigraded Cayley-Chow form} of $X$ defined by Osserman and Trager in \cite{osserman2019multigraded}. For general $0\le |\alpha|\le c-1$, typically $\mathrm{Ch}_\alpha(X)$ has codimension larger than one in $\G_\alpha$. In the remainder of this section, our goal is to establish the conditions under which $\mathrm{Ch}_\alpha(X)$ attains the expected codimension and, in this regime, compute its multidegree.

\medskip

With the above definition, we can formalize the discussion at the beginning of this section as follows.

\begin{proposition}\label{prop: multigraded vs Nash locus}
Fix a game format $\bd=(d_1,\ldots,d_n)$.
For all $i\in[n]$, let $k_i = -1+\prod_{j\neq i}d_j$. Let $\sigma_\bd$ be the mixed Segre embedding defined in \eqref{eq: segre}. Given a subvariety $Y\subseteq \PP^\bd$, the Nash locus $\kn(Y)$ decomposes into nonempty strata $\kn(Y)_\alpha$ indexed by sequences $\alpha=(\alpha_1,\dots,\alpha_n)$ with $\max(0,k_i-d_i+1)\leq \alpha_i\leq k_i$, such that for each $\alpha$, the map
\[
    \kn(Y)_\alpha\to \G_\alpha\,,\qquad G\mapsto L_G
\]
explained in the beginning of this section
is well-defined. The image of this map is $\mathrm{Ch}_\alpha(X)$. The stratum $\kn(Y)_{\alpha^*}$ for ${\alpha^*}$ with $\alpha_i^*=\max(0,k_i-d_i+1)$ is dense in $\kn(Y)$. If furthermore $k_i-d_i+1\geq 0$ for all $i\in [n]$, then each fiber of the map $\kn(Y)_{\alpha^*}\to \mathrm{Ch}_{\alpha^*}(\sigma_\bd(Y))$ is isomorphic to the product of general linear groups $\prod_{i=1}^n\mathrm{GL}({\alpha_i^*})$ as algebraic varieties.
\end{proposition}
\begin{proof}
This follows from the discussion at the beginning of this section, and by noting that, to go from $L=L_1\times \cdots\times L_n$ to a game $G$ with $L_G=L$, one needs to pick matrices $A_i\in \C^{(d_i-1)\times k_i}$ such that $L_i=\ker A_i$. We can then uniquely associate the class $[G]\in \PP(V)^n$ to such a collection of matrices $A_i$ by interpreting the entries of $A_i$ as $g^{(i)}_{j_1\cdots\,j_{i-1}\,1\,j_{i+1}\,\cdots\,j_n}-g^{(i)}_{j_1\,\cdots\,j_{i-1}\,s\,j_{i+1}\,\cdots\,j_n}$ where $j_k\in [d_k]$ specifies a column of $A_i$ and $s\in\{2,\ldots,d_i\}$ specifies a row of $A_i$. If $k_i-d_i+1=\dim(L_i)$, then choosing $A_i$ amounts to choosing a basis of $L_i$, which yields the last statement about fibers being isomorphic to products of linear spaces.
\end{proof}

\begin{example}\label{ex: diagonal n-player game equal number of options}
Suppose that $d_1=\cdots=d_n=d$ and consider the diagonal $\Delta\coloneqq\{([\pi],\ldots,[\pi])\mid[\pi]\in\PP^{d-1}\}\subseteq\PP^\bd$ of codimension $\codim\Delta=(n-1)(d-1)$. Then $\sigma_\bd(\Delta)\subseteq\PP=(\PP^{d^{n-1}-1})^n$ is contained in the diagonal of $\PP$. Applying Proposition~\ref{prop: multigraded vs Nash locus}, for $\alpha=(d^{n-1}-d,\ldots,d^{n-1}-d)$ we find that there is a dense subset $\kn(Y)_\alpha$ of $\kn(Y)$ which surjects to $\mathrm{Ch}_\alpha(\sigma_\bd(\Delta))$ with fibers isomorphic to $(\mathrm{GL}(d^{n-1}-d))^n$.
\end{example}

We now study the dimensions and multidegrees of multigraded associated varieties. Assume $\dim(X)=r$, in particular $r=k_1+\cdots+k_n-c$. Fix also $\beta=(\beta_1,\ldots,\beta_n)$ with $0\le \beta_i\le k_i$ for all $i$ and $|\beta| = r + e$ for some $e\ge 0$, and write $\alpha_i = k_i - \beta_i$ for each $i$.

\begin{proposition}\label{prop: when Chow is of expected codimension}
Under the previous notations, we have that $\mathrm{Ch}_\alpha(X)$ has the expected codimension $e$ if and only if for every nonempty $I \subseteq [n]$ we have
\begin{equation}\label{eq:genericity_conditions}
    \dim p_I(X) \ge |\beta_I|-e
\end{equation}
where $p_I(X)$ denotes the projection of $X$ onto $\prod_{i\in I}\mathbb{P}^{k_i}$.
\end{proposition}
\begin{proof}
The proof follows the argument used in \cite[Proposition 3.1]{osserman2019multigraded}. Consider the incidence correspondence
\begin{equation}\label{eq: incidence}
    I_\alpha(X) \coloneqq \{(x,L)\in X\times \G_\alpha \mid \text{$x_i\in L_i$ for all $i\in[n]$}\}\,.
\end{equation}
By definition $\mathrm{Ch}_\alpha(X)$ is the image of $I_\alpha(X)$ under projection to $\G_\alpha$.
The other projection of $I_\alpha(X)$ to $X$ makes $I_\alpha(X)$ a fiber bundle over $X$ where each fiber is isomorphic to \[
\G_\alpha'\coloneq \G(\alpha_1-1,k_1-1)\times \cdots \times \G(\alpha_n-1,k_n-1).
\]
In particular we see that $I_\alpha(X)$ is irreducible of dimension  $e_1 = r + \sum_{i=1}^n \alpha_i(k_i-\alpha_i)$.
This also implies that $\mathrm{Ch}_\alpha(X)$ is irreducible whenever $X$ is.
The dimension of the product of Grassmannians $\G_\alpha$ is $e_2 =\sum_{i=1}^n (\alpha_i+1)(k_i-\alpha_i)$. Since $e_2-e_1=|\beta|-r=e$, we see that $\codim(\mathrm{Ch}_\alpha(X))=e$ if and only if the generic fiber of the projection of $I_\alpha(X)$ to the product of Grassmannians is zero-dimensional. This happens if and only if there exists $L\in \G_\alpha$ such that $X \cap L$ is finite and nonempty. We will now show that the existence of such an $L$ is equivalent to the claimed conditions \eqref{eq:genericity_conditions}.

\medskip

First suppose that $\dim p_I(X) < |\beta_I|-e$ for some $I \subseteq [n]$.
Let $L\in G_\alpha$ be such that $X \cap L\neq\emptyset$, and write $L_I =\prod_{i \in I} L_i$, and $L_{I^c}=\prod_{i \notin I} L_i$. As $X \cap p_I^{-1} L_I\neq \emptyset$, its dimension is at least the generic fiber dimension of the map $p_I$, which is $r-\dim p_I(X)$. Therefore  
\[
    \dim(X \cap p_I^{-1} L_I)\ \geq\ r-\dim p_I(X)\ >\ r+e-|\beta_I|\ =\ |\beta_{I^c}|\,.
\]
On the other hand $\codim(p_{I^c}^{-1} L_{I^c}) = |\beta_{I^c}|$, so as $X \cap L = (X \cap p_I^{-1} L_I) \cap p_{I^c}^{-1} L_{I^c}$, we conclude that $\dim(X \cap L)>0$. Hence there exists no $L\in \G_\alpha$ such that $X\cap L$ is finite and nonempty.

\medskip

For the reverse implication, suppose that the stated inequalities \eqref{eq:genericity_conditions} are satisfied for all nonempty $I\subseteq [n]$. It follows that for any such $I$, the generic fiber dimension of $p_I$ is $r-\dim p_I(X)\leq r-|\beta|+e=|\beta_{I^c}|$. Hence there exists a point $x=(x_1,\dots,x_n) \in X$ such that for all $I\subsetneq[n]$, the fiber $p_I^{-1}(\{p_I(x)\})\cap X$ has dimension at most $|\beta_{I^c}|=\sum_{i \not \in I} (k_i-\alpha_i)$. We fix such a point $x\in X$ and claim that for $L\in \G_\alpha$ with $x_i \in L_i$ generic, the intersection $X \cap L$ is nonempty and finite. To see this, consider the incidence variety
\[
    Y \coloneqq \{(y,L)\in X\times \G_\alpha\mid y\in L, x\in L\}\,.
\]
Clearly, the image of $Y$ under the second projection is precisely the variety of all products of linear spaces $L$ which contain $x$, since for such $L$ we have $(x,L)\in Y$. This image is itself isomorphic to the product of Grassmannians $\G_\alpha'$ and hence of dimension $\sum_{i=1}^n \alpha_i(k_i-\alpha_i)$.  Our claim precisely states that the generic fiber of this second projection is finite and hence it suffices to show that $\dim(Y)\leq \sum_{i=1}^n \alpha_i(k_i-\alpha_i)$. For this note that $Y$ decomposes into the pieces
\[
    Y_I\coloneqq \{(y,L)\in Y\mid y_i=x_i \Leftrightarrow i\in I\}
\]
for all (possibly empty) $I\subseteq [n]$, and hence it suffices to bound the dimension of $Y_I$. Fix $I\subseteq [n]$ such that $Y_I$ is nonempty, then under the first projection $Y_I$ maps into $p_I^{-1}(\{p_I(x)\}) \cap X \subseteq X$ and every fiber is isomorphic to 
\[
    \prod_{i\in I}\G(\alpha_i-1,k_i-1)\times \prod_{i\notin I}\G(\alpha_i-2,k_i-2)
\]
and hence of dimension $\sum_{i \in I} \alpha_i(k_i-\alpha_i)+\sum_{i \not \in I} (\alpha_i-1)(k_i-\alpha_i)$.
We conclude that
\begin{align*}
    \dim(Y_I)\ &=\ \dim(p_I^{-1}(p_I(P)) \cap X) 
    +\sum_{i \in I} \alpha_i(k_i-\alpha_i)+\sum_{i \not \in I} (\alpha_i-1)(k_i-\alpha_i)\\
    &\leq\ \sum_{i \not \in I}(k_i-\alpha_i)+\sum_{i \in I} \alpha_i(k_i-\alpha_i)+\sum_{i \not \in I} (\alpha_i-1)(k_i-\alpha_i)\\
    &=\ \sum_i \alpha_i(k_i-\alpha_i)
\end{align*}
where the inequality comes precisely from our choice of $x$. This shows $\dim(Y)\leq \sum_i \alpha_i(k_i-\alpha_i)$ and hence finishes the proof.
\end{proof}

\begin{example}
Following up on Example~\ref{ex: diagonal n-player game equal number of options}, we show that the variety $X=\sigma_\bd(\Delta)$ satisfies the assumptions of Proposition~\ref{prop: when Chow is of expected codimension}. Observe that each component of the restriction of $\sigma_\bd$ to the diagonal $\Delta$ is equal to composing the Veronese embedding $\PP^{d-1}\hookrightarrow\PP(\mathrm{Sym}^{n-1}\C^d)$ with the inclusion $\PP(\mathrm{Sym}^{n-1}\C^d)\hookrightarrow\PP((\C^d)^{\otimes(n-1)})\cong\PP^{d^{n-1}-1}$, and the image of this composition has degree $(n-1)^{d-1}$. This means that the multidegree of $X$ is
\[
    [X]\ =\ (n-1)^{d-1}\sum_{|\gamma|=c}H^\gamma\in\frac{\Z[H_1,\ldots,H_n]}{\langle H_1^{d^{n-1}},\ldots,H_n^{d^{n-1}}\rangle}\,,\quad c\ =\ n(d^{n-1}-1)-(d-1)\,.
\]
Since we are choosing $\alpha=(d^{n-1}-d,\ldots,d^{n-1}-d)$, the expected codimension of $\mathrm{Ch}_\alpha(X)$ is $e=\codim_{\PP^\bd}(\Delta)=(n-1)(d-1)$. Furthermore, for every nonempty $I\subseteq[n]$, the projection $p_I(X)$ has dimension $d-1$. Since $\beta=(d-1,\ldots,d-1)$, we have $|\beta_I|-e=(|I|-(n-1))(d-1)\le 0$ for $I\neq [n]$, hence the condition of Proposition \ref{prop: when Chow is of expected codimension} is satisfied.
\end{example}

\begin{example}\label{ex: Nash resultants}
We now show an example of a variety $X\subseteq\PP$ and a multiindex $\alpha$ such that the inequalities in Proposition \ref{prop: when Chow is of expected codimension} are not satisfied. Again, this example connects multigraded associated varieties to Game Theory.
Let $\bd=(2,2,4)$ and consider the mixed Segre embedding $\sigma_{(2,2,4)}$ of $\PP^1\times\PP^1\times\PP^3$ into $\PP=\PP^7\times\PP^7\times\PP^3$ defined by
\[
    \sigma_{(2,2,4)}([\pi^{(1)}],[\pi^{(2)}],[\pi^{(3)}]) \coloneqq ([\pi^{(2)}\otimes\pi^{(3)}],[\pi^{(1)}\otimes\pi^{(3)}],[\pi^{(1)}\otimes\pi^{(2)}])\,.
\]
Consider its image $X=\im\sigma_{(2,2,4)}$. Fix the multiindex $\alpha=(6,6,0)$. For a generic product $L=L_1\times L_2\times L_3$ with $\dim L_i=\alpha_i$, we have $L\cap X=\emptyset$. The multigraded associated variety $\mathrm{Ch}_{(6,6,0)}(X)$ is precisely the locus of $L$'s for which $L\cap X\neq\emptyset$. Pulling back this condition under the mixed Segre embedding $\sigma_{(2,2,4)}$, we obtain precisely the locus of coefficients $a_{ij}$, $b_{ij}$, and $c_{ij}^{(s)}$ with $s\in[3]$ such that the system of five equations
\begin{align*}
    a_{11}\pi_1^{(2)}\pi_1^{(3)}+\cdots+a_{24}\pi_2^{(2)}\pi_4^{(3)} &= 0\,,\\
    b_{11}\pi_1^{(1)}\pi_1^{(3)}+\cdots+b_{24}\pi_2^{(1)}\pi_4^{(3)} &= 0\,,\\
    c_{11}^{(s)}\pi_1^{(1)}\pi_1^{(2)}+\cdots+c_{22}^{(s)}\pi_2^{(1)}\pi_2^{(2)} &= 0\,,\quad\forall\,i\in[3]
\end{align*}
admits at least one solution in $\PP^1\times\PP^1\times\PP^3$, namely the Nash equilibrium scheme $\kz_{[G]}$ is nonempty. This locus is called {\em Nash resultant variety} in \cite{abo2026vector}. Recall that the previous polynomial system is induced by a game $G$ of unbalanced format $(2,2,4)$, and its zero locus is the Nash equilibrium scheme $\kz_G$, which is empty for generic payoff tensors $G^{(1)},G^{(2)},G^{(3)}$.
In \cite[Example 3.33]{abo2026vector}, it is shown that the Nash resultant variety is a hypersurface of multidegree $(0,0,6)$, namely its equation, called {\em Nash resultant}, depends only on the entries of the third payoff tensor $G^{(3)}$. Equivalently, the equation of $\mathrm{Ch}_{(6,6,0)}(X)$ depends only on the coordinates of $L_3$.
However, since $r=5$ and $\alpha=(6,6,0)$, we get $\beta=(1,1,3)$ and the expected codimension of $\mathrm{Ch}_{(6,6,0)}(X)$ is $e=|\beta|-r=0$. The expected codimension $e=0$ does not match the actual codimension of $\mathrm{Ch}_{(6,6,0)}(X)$ because, choosing $I=\{3\}$, then $\dim p_I(X)=\dim(\PP^1\times\PP^1)=2<3=\beta_3$, hence Proposition \ref{prop: when Chow is of expected codimension} cannot be applied to this example. Alternatively the proof of Proposition \ref{prop: when Chow is of expected codimension} reveals that the codimension of $\mathrm{Ch}_{(6,6,0)}(X)$ is unexpected whenever there exists no linear space $L\in \G_\alpha$ which has finite nonempty intersection with $X$. In our case any candidate $L=L_1\times L_2\times \{\pi^{(1)}\otimes \pi^{(2)}\}$ will intersect $X$ in a line.

More in general, for any game format $\bd$ such that $d_1\le\cdots\le d_n$ and $d_n-1>\sum_{i=1}^{n-1}(d_i-1)$ when $\alpha_i=k_i-d_i+1$ (assuming this is nonnegative), then $\beta_i=k_i-\alpha_i=d_i-1$ for all $i\in[n]$, therefore $|\beta|=r$ and the expected dimension of $\mathrm{Ch}_\alpha(X)$ is zero. Despite this, Proposition \ref{prop: when Chow is of expected codimension} is not applicable because $\dim p_{\{n\}}(X)=\sum_{i=1}^{n-1}(d_i-1)<d_n-1=\beta_n$. The correct codimension of $\mathrm{Ch}_\alpha(X)$ is $d_n-1-\sum_{i=1}^{n-1}(d_i-1)$ by \cite[Theorem 3.31]{abo2026vector}.
\end{example}

We now turn our attention to the multidegrees of multigraded associated varieties. On each Grassmannian $\G_i$ we have the tautological sequence
\begin{equation}\label{eq: ith tautological sequence}
    0\to S_i \to \ko_{\G_i}^{\oplus (k_i+1)} \to Q_i \to 0\,,
\end{equation}
with $\mathrm{rank}(S_i)=\alpha_i+1$ and $\mathrm{rank}(Q_i)=k_i-\alpha_i$.
We denote by the same symbols $S_i$ and $Q_i$ their pullbacks to $\G_\alpha$.
Let $H_i\coloneqq c_1(\ko_{\PP^{k_i}}(1))$ on $\PP^{k_i}$, and also its pullback to $\PP$.
Then
\[
    A^*(\PP)\cong \frac{\Z[H_1,\dots,H_n]}{(H_1^{k_1+1},\dots,H_n^{k_n+1})}\,.
\]
The multidegree of $X$ is
\begin{equation}\label{eq: multidegree X}
    [X]\ =\ \sum_{|\gamma|=c} \delta_\gamma(X)\, H_1^{\gamma_1}\cdots H_n^{\gamma_n}\,.
\end{equation}
Consider the incidence correspondence $I_\alpha(\PP)$ introduced in \eqref{eq: incidence}. Define $\pi\colon I_\alpha(\PP)\to \G_\alpha$ and $p\colon I_\alpha(\PP)\to\PP$ as the natural projections.
There is a canonical identification $I_\alpha(\PP) \cong \PP(S_1)\times_{\G_\alpha} \cdots \times_{\G_\alpha}\PP(S_n)$, and we write $\zeta_i \coloneqq c_1(\ko_{\PP(S_i)}(1))\in A^1(I_\alpha)$ so that $p^*(H_i)=\zeta_i$; see \cite[Section 9.3.1]{EH}. 

\begin{example}
\label{ex: Chow grass lines P3}
The method {\tt productGrassmannChowRing} available in the file {\tt nashLoci.m2} allows to compute the Chow ring $A^*[\G_\alpha]$ given the vectors of affine dimensions $(\alpha_1+1,\ldots,\alpha_n+1)$ and $(k_1+1,\ldots,k_n+1)$ as input. In the following example, we consider the Chow ring of $\G(1,3)\times\G(1,3)$. Recall that, applying \cite[Theorem 5.26]{EH},
\[
    A^*(\G(1,3)) = \frac{\Z[\sigma_1,\sigma_2]}{\langle\sigma_1^3-2\sigma_1\sigma_2,\sigma_1^2\sigma_2-\sigma_2^2\rangle}
\]
where $\sigma_1$ and $\sigma_2$ are the Chern classes of the universal subbundle $S$ in the exact sequence \eqref{eq: tautological sequence}. Using the same short exact sequence, we see $c(Q)=1/c(S)=1-\sigma_1+\sigma_1^2-\sigma_2$. Hence
\begin{equation}\label{eq: chowring rep for 2G(1,3)}
    A^*(\G(1,3)\times \G(1,3)) = \frac{\Z[\sigma_1,\sigma_2,\tau_1,\tau_2]}{\langle\sigma_1^3-2\sigma_1\sigma_2,\sigma_1^2\sigma_2-\sigma_2^2,\tau_1^3-2\tau_1\tau_2,\tau_1^2\tau_2-\tau_2^2\rangle}\,.
\end{equation}
Differently from above, in the script displayed in Figure~\ref{fig: productGrassmannChowRing}, we present $A^*(\G(1,3)\times \G(1,3))$ using the variables $q_{ij}=c_j(Q_i)$. In the previous example we have $q_{11}=-\sigma_1$, $q_{21}=-\tau_1$, $q_{12}=\sigma_1^2-\sigma_2$, and $q_{22}=\tau_1^2-\tau_2$. Furthermore, we can also integrate a top-degree class in $A^*[\G_\alpha]$. For example, consider the class $q_{11}q_{21}$, namely the class of the locus of pairs of lines in $\PP^3\times\PP^3$ meeting a fixed pair of lines. Multiplying it by $q_{11}^3q_{21}^3$, we obtain a top-degree class in $A^*[\G_\alpha]$, and we can apply the method {\tt integrateProductGrassmann}. The output 4 should be read as $2\cdot 2$, where $2$ is the degree of the Schubert variety of lines in $\PP^3$ meeting a given line.
\end{example}

\begin{figure}[h]
\centering
\begin{tcolorbox}[size=fbox,width=\linewidth,colback=blue!5!white,colframe=blue!75!black]
\begin{Verbatim}[fontsize=\scriptsize,commandchars=\\\{\}]
i1 : load "nashLoci.m2"

i2 : K = {2,2}; N = {4,4};

i4 : intRing = productGrassmannChowRing(K,N);

i5 : describe intRing

                                           ZZ[q   ..q   ]
                                               1,1   2,2
o5 = ------------------------------------------------------------------------------------------
         3                 4       2          2       3                 4       2          2
     (- q    + 2q   q   , q    - 3q   q    + q   , - q    + 2q   q   , q    - 3q   q    + q   )
         1,1     1,1 1,2   1,1     1,1 1,2    1,2     2,1     2,1 2,2   2,1     2,1 2,2    2,2

i6 : clX = q_(1,1)*q_(2,1)

o6 = q   q
      1,1 2,1

o6 : intRing

i7 : integrateProductGrassmann(K,N,clX*q_(1,1)^3*q_(2,1)^3)

o7 = 4
\end{Verbatim}
\end{tcolorbox}
\caption{Usage of the {\tt nashLoci.m2} methods {\tt productGrassmannChowRing} and {\tt integrateProductGrassmann}.}\label{fig: productGrassmannChowRing}
\end{figure}

\begin{proposition}\label{prop: class of multigraded associated variety}
Assume that the inequalities in Proposition~\ref{prop: when Chow is of expected codimension} are satisfied. Then
\[
    [\mathrm{Ch}_\alpha(X)] \ =\  \pi_*\,p^*[X]\in A^*[\G_\alpha]\,.
\]
Equivalently, using the multidegree expansion \eqref{eq: multidegree X},
\begin{equation}\label{eq: deg multigraded associated variety}
    [\mathrm{Ch}_\alpha(X)]\ =\ \sum_{\substack{|\gamma|=c\\\gamma_i\ge\alpha_i\ \forall i}} \delta_\gamma(X)\ \prod_{i=1}^n s_{\gamma_i-\alpha_i}(S_i)\ =\ \sum_{\substack{|\gamma|=c\\\gamma_i\ge \alpha_i\ \forall i}}\delta_\gamma(X)\ \prod_{i=1}^n c_{\gamma_i-\alpha_i}(Q_i)\,.
\end{equation}
\end{proposition}
\begin{proof}
From the incidence definition, as a cycle class one has $[\mathrm{Ch}_\alpha(X)]=\pi_*p^*[X]$.
Using \eqref{eq: multidegree X} and $p^*(H_i)=\zeta_i$ gives
\[
    \pi_*p^*[X]\ =\ \sum_{|\gamma|=c}\delta_\gamma(X)\ \pi_*\!\left(\prod_{i=1}^n \zeta_i^{\gamma_i}\right)\,.
\]
Since $\pi$ is a fiber product of projective bundles $\PP(S_i)\to \G_\alpha$ of relative dimension $\alpha_i$, projective bundle pushforward yields $(\pi_i)_*\bigl(\zeta_i^{\alpha_i+j}\bigr)=s_j(S_i)$ ($j\ge 0$), $(\pi_i)_*(\zeta_i^m)=0$ for $m<\alpha_i$.
Thus only terms with $\gamma_i\ge \alpha_i$ contribute and $\pi_*\!\left(\prod_{i=1}^n \zeta_i^{\gamma_i}\right) = \prod_{i=1}^n s_{\gamma_i-\alpha_i}(S_i)$.
Finally $s(S_i)=c(S_i)^{-1}$ and $c(S_i)c(Q_i)=1$ from the Whitney sum formula applied to \eqref{eq: ith tautological sequence}, hence $s_{t}(S_i)=c_t(Q_i)$ for all $t$.
\end{proof}

\begin{remark}
Observe that, when $|\alpha|=c-1$, then \eqref{eq: deg multigraded associated variety} simplifies to
\[
    [\mathrm{Ch}_\alpha(X)]\ =\ \sum_{i=1}^n \delta_{\alpha+e_i}(X)\,\prod_{j=1}^n c_{\delta_{ij}}(Q_j)\ =\ \sum_{i=1}^n \delta_{\alpha+e_i}(X)\,c_1(Q_i)\,,
\]
where $e_i$ is the $i$-th standard basis vector in $\mathbb{N}^n$ and $\delta_{ij}$ is the Kronecker delta. This agrees with \cite[Theorem 1.2]{osserman2019multigraded}.
\end{remark}

\begin{example}
Consider $n=3$, $k=1=k_2=k_3=3$, $\alpha=(2,2,2)$ and let $X=\sigma_{(1,1,1)}(Y)$, where $\sigma_{(1,1,1)}$ is the mixed Segre embedding of $(\PP^1)^3$ in $(\PP^3)^3$ introduced at the beginning of Section~\ref{sec: multigraded associated varieties} and $Y$ is a generic complete intersection variety in $(\PP^1)^3$ of codimension $2$, in particular its class in $A^*((\PP^1)^3)=\Z[h_1,h_2,h_3]/\langle h_1^2,h_2^2,h_3^2\rangle$ is
\begin{align*}
    [Y] &= (a_{11}h_1+a_{12}h_2+a_{13}h_3)(a_{21}h_1+a_{22}h_2+a_{23}h_3)\\
    &= (a_{12}a_{21}+a_{11}a_{22})h_1h_2+(a_{13}a_{21}+a_{11}a_{23})h_1h_3+(a_{13}a_{22}+a_{12}a_{23})h_2h_3\,.
\end{align*}
Hence we can use only the three numbers $b_{12}\coloneqq a_{12}a_{21}+a_{11}a_{22}$, $b_{13}\coloneqq a_{13}a_{21}+a_{11}a_{23}$ and $b_{23}\coloneqq a_{13}a_{22}+a_{12}a_{23}$ to define the class of $Y$. Observe that in this case $c=\codim_\PP(Y)=\codim_\PP(\sigma_{(1,1,1)}((\PP^1)^3))+2=6+2=8$ and $\G_\alpha=\G(2,3)^3=((\PP^3)^*)^3$.
In this case, the tautological exact sequence for every factor $\G(2,3)$ is (after identifying $\G(2,3)$ with $\PP^3$)
\begin{equation}
    0\to \ko_{\PP^3}(-1) \to \ko_{\PP^3}^{\oplus 4} \to Q \to 0
\end{equation}
where $Q$ has rank $1$. Calling $q_1=c_1(\ko_{\PP^3}(1))$, then applying Whitney sum formula we get $c(Q)=\frac{1}{1-q_1}=1+q_1+q_1^2+q_1^3$. Note also that $Q=T_{\PP^3}(-1)$. We repeat the same for the other factors of $\G_\alpha$, using variables $q_2$ and $q_3$.
Applying Proposition~\ref{prop: when Chow is of expected codimension} and Proposition~\ref{prop: class of multigraded associated variety}, then
\[
    [\mathrm{Ch}_\alpha(X)]\ =\ \sum_{\substack{|\gamma|=8\\3\ge \gamma_i\ge 2\ \forall i}}\delta_\gamma(X)\ \prod_{i=1}^3 q_i^{\gamma_i-\alpha_i}\,.
\]
We compute all $\delta_\gamma(X)$, which are the coefficients of the multidegree
\[
    [X]=\sum_{\substack{|\gamma|=8\\\gamma_i\le 3\ \forall i}}\delta_\gamma(X)H_1^{\gamma_1}H_2^{\gamma_2}H_3^{\gamma_3}\in A^*((\PP^3)^3)=\frac{\Z[H_1,H_2,H_3]}{\langle H_1^4,H_2^4,H_3^4\rangle}\,.
\]
For this, we note that intersecting with $H_1$, by the push-pull formula, we obtain
\[
    \delta_{(2,3,3)}(X)=\int_{(\PP^3)^3}[X]\,H_1\ =\ \int_{(\PP^1)^3}[Y]\,(h_2+h_3)\ =\ b_{12}+b_{13}\,.
\]
Analogously we get $\delta_{(3,2,3)}(X)=b_{12}+b_{23}$ and $\delta_{(3,3,2)}(X)=b_{13}+b_{23}$. Now substituting this into the formula for $[\mathrm{Ch}_\alpha(X)]$ we get
\[
    [\mathrm{Ch}_{(2,2,2)}(X)]\ =\ (b_{12}+b_{13})\,q_2q_3+(b_{12}+b_{23})\,q_1q_3+(b_{13}+b_{23})\,q_1q_2\,.
\]
The previous identity can be verified using the script in Figure~\ref{fig: degAssVar complete intersection}.
\end{example}
\begin{figure}[h]
\centering
\begin{tcolorbox}[size=fbox,width=\linewidth,colback=blue!5!white,colframe=blue!75!black]
\begin{Verbatim}[fontsize=\small,commandchars=\\\{\}]
A = ZZ[b_(1,2),b_(1,3),b_(2,3)];
R = A[t_1,t_2,t_3]/ideal(t_1^4,t_2^4,t_3^4);
degX = (b_(1,2)+b_(1,3))*t_1^2*t_2^3*t_3^3+(b_(1,2)+b_(2,3))*t_1^3*t_2^2*t_3^3+
(b_(1,3)+b_(2,3))*t_1^3*t_2^3*t_3^2;
mdeg = degAssVar(\{3,3,3\},\{4,4,4\},degX)
\end{Verbatim}
\end{tcolorbox}
\caption{Commands for {\tt degAssVar} given the multidegree of $X$ with symbolic coefficients as input.}\label{fig: degAssVar complete intersection}
\end{figure}

\begin{example}
Following up on the previous example, consider the diagonal $Y=\Delta\subseteq(\PP^1)^3$ and let $X=\sigma_{(1,1,1)}(Y)\subseteq(\PP^3)^3$. In particular $[Y] = h_1h_2+h_1h_3+h_2h_3$ and  $\mathrm{Ch}_{(2,2,2)}(X)$ is a two-codimensional variety of multidegree $[\mathrm{Ch}_{(2,2,2)}(X)] = 2\,q_2q_3+2\,q_1q_3+2\,q_1q_2$. We computed the equations of $\mathrm{Ch}_{(2,2,2)}(X)$ in \texttt{Macaulay2}, see the file \texttt{ExDiagonal.m2} available at \cite{sodomaco2026supplementary}. We consider coordinates $[a_{00},a_{01},a_{10},a_{11}]$, $[b_{00},b_{01},b_{10},b_{11}]$ and $[c_{00},c_{01},c_{10},c_{11}]$ for the three factors in $\G_{(2,2,2)}=((\PP^3)^*)^3$.
In particular $\mathrm{Ch}_{(2,2,2)}(X)$ encodes all products of planes $L=L_1\times L_2\times L_3\in\G_{(2,2,2)}$ such that $L\cap X\neq\emptyset$. Equivalently, we seek conditions on $a_{ij}$, $b_{ij}$ and $c_{ij}$ under which the system in $(\PP^1)^3$
\begin{align}\label{eq: tmNe 2x2x2 game}
\begin{split}
    a_{00}y_0z_0+a_{01}y_0z_1+a_{10}y_1z_0+a_{11}y_1z_1 &= 0\,,\\
    b_{00}x_0z_0+b_{01}x_0z_1+b_{10}x_1z_0+b_{11}x_1z_1 &= 0\,,\\
    c_{00}x_0y_0+c_{01}x_0y_1+c_{10}x_1y_0+c_{11}x_1y_1 &= 0
\end{split}
\end{align}
admits at least one solution in $\Delta$.
We verified that the ideal of $\mathrm{Ch}_{(2,2,2)}(X)$ has $7$ minimal generators of multidegrees $(1,1,1)$, $(0,2,2)$, $(2,0,2)$, $(2,2,0)$, $(2,1,1)$, $(1,2,1)$, $(1,1,2)$, displayed below:
\[
\resizebox{\textwidth}{!}{
$\begin{aligned}
\mathrm{det} &= 
 \begin{vmatrix}
     a_{00} & a_{01}+a_{10} & a_{11}\\[2pt]
     b_{00} & b_{01}+b_{10} & b_{11}\\[2pt]
     c_{00} & c_{01}+c_{10} & c_{11}
 \end{vmatrix}\,,\\
\mathrm{Det}_{1} &=\left(
 \begin{vmatrix}
 b_{00} & \frac{(c_{01}+c_{10})}{2}\\[2pt]
 \frac{(b_{01}+b_{10})}{2} & c_{11}
 \end{vmatrix}
 +
 \begin{vmatrix}
 \frac{(b_{01}+b_{10})}{2} & c_{00}\\[2pt]
 b_{11} & \frac{(c_{01}+c_{10})}{2}
 \end{vmatrix}
 \right)^2 - 4
 \begin{vmatrix}
 b_{00} & c_{00}\\[2pt]
 \frac{(b_{01}+b_{10})}{2} & \frac{(c_{01}+c_{10})}{2}
 \end{vmatrix}
 \cdot
 \begin{vmatrix}
 \frac{(b_{01}+b_{10})}{2} & \frac{(c_{01}+c_{10})}{2}\\[2pt]
 b_{11} & c_{11}
 \end{vmatrix}\,,\\
\mathrm{Det}_{2} &=\left(
 \begin{vmatrix}
 a_{00} & \frac{(c_{01}+c_{10})}{2}\\[2pt]
 \frac{(a_{01}+a_{10})}{2} & c_{11}
 \end{vmatrix}
 +
 \begin{vmatrix}
 \frac{(a_{01}+a_{10})}{2} & c_{00}\\[2pt]
 a_{11} & \frac{(c_{01}+c_{10})}{2}
 \end{vmatrix}
 \right)^2 - 4
 \begin{vmatrix}
 a_{00} & c_{00}\\[2pt]
 \frac{(a_{01}+a_{10})}{2} & \frac{(c_{01}+c_{10})}{2}
 \end{vmatrix}
 \cdot
 \begin{vmatrix}
 \frac{(a_{01}+a_{10})}{2} & \frac{(c_{01}+c_{10})}{2}\\[2pt]
 a_{11} & c_{11}
 \end{vmatrix}\,,\\
\mathrm{Det}_{3} &=\left(
 \begin{vmatrix}
 a_{00} & \frac{(b_{01}+b_{10})}{2}\\[2pt]
 \frac{(a_{01}+a_{10})}{2} & b_{11}
 \end{vmatrix}
 +
 \begin{vmatrix}
 \frac{(a_{01}+a_{10})}{2} & b_{00}\\[2pt]
 a_{11} & \frac{(b_{01}+b_{10})}{2}
 \end{vmatrix}
 \right)^2 - 4
 \begin{vmatrix}
 a_{00} & b_{00}\\[2pt]
 \frac{(a_{01}+a_{10})}{2} & \frac{(b_{01}+b_{10})}{2}
 \end{vmatrix}
 \cdot
 \begin{vmatrix}
 \frac{(a_{01}+a_{10})}{2} & \frac{(b_{01}+b_{10})}{2}\\[2pt]
 a_{11} & b_{11}
 \end{vmatrix}\,,\\
 s_1 &=
 \begin{vmatrix}
 b_{00} & b_{01}+b_{10} & b_{11} & 0\\[2pt]
 0 & c_{00} & c_{01}+c_{10} & c_{11}\\[2pt]
 a_{00} & a_{01}+a_{10} & a_{11} & 0\\[2pt]
 0 & a_{00} & a_{01}+a_{10} & a_{11}
 \end{vmatrix}\,,\ 
 s_2 =
 \begin{vmatrix}
 a_{00} & a_{01}+a_{10} & a_{11} & 0\\[2pt]
 0 & c_{00} & c_{01}+c_{10} & c_{11}\\[2pt]
 b_{00} & b_{01}+b_{10} & b_{11} & 0\\[2pt]
 0 & b_{00} & b_{01}+b_{10} & b_{11}
 \end{vmatrix}\,,\ 
 s_3 =
 \begin{vmatrix}
 a_{00} & a_{01}+a_{10} & a_{11} & 0\\[2pt]
 0 & b_{00} & b_{01}+b_{10} & b_{11}\\[2pt]
 c_{00} & c_{01}+c_{10} & c_{11} & 0\\[2pt]
 0 & c_{00} & c_{01}+c_{10} & c_{11}
 \end{vmatrix}\,.
\end{aligned}$}
\]
The condition $\det=0$ imposes that the three equations in \eqref{eq: tmNe 2x2x2 game} are linearly dependent after restricting to $\Delta$. The polynomial $\mathrm{Det}_1$ is the {\em $2\times 2\times 2$ hyperdeterminant}, more specifically the resultant of the second and third equations in \eqref{eq: tmNe 2x2x2 game} after restricting to $\Delta$, equivalently the resultant of two binary quadrics. Similarly for $\mathrm{Det}_2$ and $\mathrm{Det}_3$.
\end{example}

\begin{example}\label{ex: GR(1,3)}
Following \cite[Example 2.8]{pratt2026multigraded}, let $X\subseteq \PP^3\times \PP^3$ be the conormal variety of the quadratic surface $Q=V(x_1x_2-x_0x_3)\subseteq \PP^3$. The multidegree of $X$ is $[X]=2\,H_1^3H_2+2\,H_1^2H_2^2+2\,H_1H_2^3$, the three coefficients $(2,2,2)$ are hence the polar degrees of the quadratic surface $Q$. The codimension of $X$ is 4, hence any choice $\alpha=(\alpha_1,\alpha_2)$ with $|\alpha|\leq 3$ will give a non-trivial multigraded Chow variety. In particular in our notation we have $r=\dim(X)=2,c=\codim(X)=4,e=|\beta|-r=6-|\alpha|-r\geq 1$.
    
First notice that the two projections $p_{\{1\}}(X),p_{\{2\}}(X)$ give the quadratic surface $Q$ and its dual $Q^\vee$ which is again a quadratic surface. In particular we have $\dim(p_{\{1\}}(X))=\dim(p_{\{2\}}(X))=2$. For every choice of index $\alpha$ we have that $\beta_i=3-\alpha_i\leq 3$ and hence $|\beta_{\{i\}}|-e\leq 2$ for $i=1,2$. Therefore, the conditions of Proposition \ref{prop: when Chow is of expected codimension} are satisfied for every choice of $\alpha$, and we can use Proposition \ref{prop: class of multigraded associated variety} to compute the class $[\mathrm{Ch}_\alpha(X)]$. Using the representation \eqref{eq: chowring rep for 2G(1,3)} of $A^*(\mathbb{G}(1,3)\times \mathbb{G}(1,3))$ and focusing on the most interesting case $\alpha=(1,1)$, we get $[\mathrm{Ch}_{(1,1)}(X)]=2(\sigma_1^2-\sigma_2)+2\sigma_1\tau_1+2(\tau_1^2-\tau_2)$.
\end{example}

Each factor $\G_i=\G(\alpha_i,k_i)$ has its Pl\"ucker embedding; let $q_i \coloneqq c_1(\det Q_i)=c_1(Q_i)\in A^1(\G_i)$, and also denote by $q_i$ its pullback to $\G_\alpha$.
The product $\G_\alpha$ embeds into a product of projective spaces via the Pl\"ucker embeddings. It is then natural to record the \emph{multidegree} with respect to the line bundles $\ko(q_1),\dots,\ko(q_n)$. If $Y\subseteq \G_\alpha$ is a subvariety of dimension $m$, define its {\em Pl\"ucker multidegrees} by
\[
    \deg_{m_1,\dots,m_n}(Y)\coloneqq \int_{\G_\alpha} [Y]\ \prod_{i=1}^n q_i^{m_i},\qquad m_1+\cdots+m_n=m\,.
\]

\begin{proposition}[Pl\"ucker multidegrees of $\mathrm{Ch}_\alpha(X)$]\label{prop:multideg-multi}
Let $m\coloneqq\dim \mathrm{Ch}_\alpha(X)$ (expected or actual). For any $m_1,\dots,m_n\ge 0$ with $\sum_{i=1}^n m_i=m$,
\[
    \deg_{m_1,\dots,m_n}(\mathrm{Ch}_\alpha(X))\ =\ \sum_{\substack{|\gamma|=c\\ b_i\ge \alpha_i}} \delta_\gamma(X)\ \prod_{i=1}^n\int_{\G_i}c_{\,\gamma_i-\alpha_i}(Q_i)\,q_i^{m_i}\,.
\]
\end{proposition}
\begin{proof}
Insert the class formula from Proposition~\ref{prop: class of multigraded associated variety} into the definition of multidegrees.
Since $A^*(\G_\alpha)\cong \bigotimes_{i=1}^k A^*(\G_i)$, the remaining integrals are products of Schubert numbers on the individual Grassmannians.
\end{proof}

\begin{example}
We continue Example \ref{ex: GR(1,3)} and compute the Pl\"ucker multidegrees of $\mathrm{Ch}_{(1,1)}(X)$ where $X$ is the conormal variety of the quadratic surface $Q=V(x_1x_2-x_0x_3)$. First we need to understand $h_1,h_2$ and the degree map $\int_{\G_{(1,1)}}$ in the representation \eqref{eq: chowring rep for 2G(1,3)} of $A^*(\G(1,3)\times\G(1,3))$. As computed already in Example \ref{ex: GR(1,3)}, we have  $h_1=c_1(\det Q_1)=c_1(Q_1)=-\sigma_1,h_2=c_1(\det Q_2)=c_1(Q_2)=-\tau_1$. Furthermore, the degree map is the unique map $A^8(\G_{(1,1)})$ which sends the monomial $\sigma_2^2\tau_2^2$ to one. This can be seen, for example, from the equality
\[
    4=\deg\G_{(1,1)}=\int_{\G_{(1,1)}}q_1^4q_2^4=\int_{\G_{(1,1)}}\sigma_1^4\tau_1^4=\int_{\G_{(1,1)}}4\sigma_2^2\tau_2^2=4\int_{\G_{(1,1)}}\sigma_2^2\tau_2^2\,,
\]
where we computed the degree on the left as the degree of the image under the Pl\"ucker embedding. We get
\begin{align*}
    \deg_{2,4}(\mathrm{Ch}_{(1,1)}(X)) &= \int_{\G_{(1,1)}}2((\sigma_1^2-\sigma_2)+\sigma_1\tau_1+(\tau_1^2-\tau_2))(-\sigma_1)^2(-\tau_1)^4\\
    &= \int_{\G_{(1,1)}}2\sigma_1^4\tau_1^4-2\sigma_1^2\sigma_2\tau_1^4=2\int_{\G_{(1,1)}}\sigma_2^2(2\tau_2^2)=4\\
    \deg_{4,2}(\mathrm{Ch}_{(1,1)}(X)) &= \int_{\G_{(1,1)}}2((\sigma_1^2-\sigma_2)+\sigma_1\tau_1+(\tau_1^2-\tau_2))(-\sigma_1)^4(-\tau_1)^2\\
    &= \int_{\G_{(1,1)}}2\sigma_1^4\tau_1^4-2\tau_1^2\tau_2\sigma_1^4=2\int_{\G_{(1,1)}}\tau_2^2(2\sigma_2^2)=4\\
    \deg_{3,3}(\mathrm{Ch}_{(1,1)}(X)) &= \int_{\G_{(1,1)}}2((\sigma_1^2-\sigma_2)+\sigma_1\tau_1+(\tau_1^2-\tau_2))(-\sigma_1)^3(-\tau_1)^3\\
    &= 2\int_{\G_{(1,1)}}2\sigma_1^4\tau_1^4=2\deg\G_{(1,1)}=8\,.
\end{align*}
In particular $\deg_{(1,1)}(\mathrm{Ch}_\alpha(X))$ the class of the image of $\mathrm{Ch}_{(1,1)}(X)$ under the coordinatewise Pl\"ucker embedding is $4\,T_1^3T_2+8\,T_1^2T_2^2+4\,T_1T_2^3$.
We confirmed this computation using the function \texttt{pluckerDegAssVar} of the supplementary software \texttt{nashLoci.m2}; see Figure~\ref{fig: pluckerDegAssVar}.
In particular, the ideal of $\mathrm{Ch}_{(1,1)}(X)$ under the product of two Pl\"ucker embeddings is minimally generated by 3 quadrics and 10 polynomials of total degree 4. Two of the quadrics are simply the Pl\"ucker relations in the two factors of bidegree $(2,0)$ and $(0,2)$ respectively, the remaining quadric in the Pl\"ucker coordinates $x_{01},\ldots,x_{23}$ and $y_{01},\ldots,y_{23}$
\[
    x_{01}y_{01}-x_{02}y_{02}-x_{12}y_{03}-x_{03}y_{12}-x_{13}y_{13}+x_{23}y_{23}
\]
has bidegree $(1,1)$, and all 10 degree 4 relations have bidegree $(2,2)$.
\end{example}
\begin{figure}[h]
\centering
\begin{tcolorbox}[size=fbox,width=\linewidth,colback=blue!5!white,colframe=blue!75!black]
\begin{Verbatim}[fontsize=\small,commandchars=\\\{\}]
load "nashLoci.m2"
R = QQ[x_0..x_3]**QQ[y_0..y_3];
xx = matrix\{\{x_0..x_3\}\}; yy = matrix\{\{y_0..y_3\}\};
IQ = ideal(x_0*x_3-x_1*x_2);
IsingQ = radical ideal singularLocus IQ;
jacQ = diff(xx, transpose gens IQ)
conQ = saturate(IQ + minors(codim IQ + 1, jacQ||yy), IsingQ);
pluckerDegAssVar(\{2,2\},\{4,4\},conQ)
\end{Verbatim}
\end{tcolorbox}
\caption{Computation of the Pl\"ucker multidegree of $\mathrm{Ch}_{(1,1)}(X)$ for the conormal variety $X$ of a nonsingular quadric in $\PP^3$, whose vanishing ideal is denoted by \texttt{conQ} in the script.}\label{fig: pluckerDegAssVar}
\end{figure}

\bibliographystyle{alphaurl}
\bibliography{biblio}

\bigskip \medskip \bigskip

\noindent
\small {\bf Authors' addresses:}
\smallskip

\noindent Luca Sodomaco\\
Max Planck Institute for Mathematics in the Sciences, Leipzig, Germany\\
\url{luca.sodomaco@mis.mpg.de}

\medskip

\noindent Julian Weigert\\
Max Planck Institute for Mathematics in the Sciences, Leipzig, Germany\\
Mathematisches Institut, Georg-August Universit\"at G\"ottingen, Germany\\ \url{julian.weigert@mis.mpg.de}

\end{document}